\documentclass{article}
\usepackage[utf8]{inputenc}
\usepackage{amsmath, amsthm, amssymb}
\usepackage{amsfonts, enumitem}
\usepackage{array}
\usepackage{float}
\usepackage{setspace}
\newcommand{\white}{\vspace{4mm}}

\usepackage{graphicx}
\usepackage{tikz}
\usetikzlibrary{positioning}
\usepackage{tikz-cd}
\usepackage{mathtools}
\usepackage{geometry}
\newcommand{\N}{\mathbb{N}}

\newcommand{\C}{\mathbb{C}}

\newcommand{\Z}{\mathbb{Z}}

\newcommand{\B}{\mathcal{B}}
\newcommand{\J}{\mathcal{J}}
\newcommand{\A}{\mathcal{A}}
\newcommand{\X}{\mathcal{X}}

\newcommand{\ith}{^{\text{th}}}
\newcommand{\st}{^{\text{st}}}

\theoremstyle{definition}
\newtheorem{theorem}{Theorem}[section]
\newtheorem{corollary}[theorem]{Corollary}
\newtheorem{lemma}[theorem]{Lemma}
\newtheorem{proposition}[theorem]{Proposition}

\title{Minimal relations for the balanced algebra}
\author{Erika Pirnes}

\begin{document}

\maketitle

\noindent\rule{\textwidth}{1pt}

\vspace{-2mm}
\paragraph{Abstract:} Motivated by a problem in graph theory, this article introduces an algebra called the balanced algebra. This algebra is defined by generators and relations, and the main goal is to find a minimal set of relations for it.
\vspace{2mm}

\noindent\rule{\textwidth}{1pt}

\section{Introduction}

This article is about an algebra $\B$ called the \textit{balanced algebra}. The algebra is related to a problem that comes up in algebraic graph theory. The algebra $\B$ is defined via generators and relations, and the main goal of the article is to find a minimal set of relations for $\B$.

\paragraph{Informal explanation:} First think about all possible ``words'' in two letters $L$ and $R$. For example, $L$, $LR$, and $LLRLR$ are words. A word is called balanced if it contains equal numbers of both letters. Among the words above, only $LR$ is balanced. The elements of the balanced algebra $\B$ are linear combinations of words, for example, $LR$, or $2RR+5LLR$. Moreover, any time two balanced words appear next to each other inside a word, they may be swapped and the resulting word is considered to be the same element in $\B$ as the original word. For example, the words $LRRL$ and $RLLR$ correspond to the same element in $\B$, because they can be obtained from each other by swapping the two balanced words $LR$ and $RL$.

As another example, the words $LRLLRR=(LR)(LLRR)$ and $LLRRLR=(LLRR)(LR)$ correspond to the same element in $\B$; they can be obtained from each other by swapping $LR$ and $LLRR$. By writing the words as $LRLLRR=L(RL)(LR)R$ and $LLRRLR=L(LR)(RL)R$, it can be seen that the words can also be obtained from each other by swapping $LR$ and $RL$. As a generalization of this example, any time $LR$ and $LLRR$ appear next to each other inside a word, swapping $LR$ and $RL$ (as in the example) yields the same result as swapping $LR$ and $LLRR$. Thus it can be said that swapping $LR$ and $LLRR$ follows from swapping $LR$ and $RL$. Informally, the goal of this article is to find a ``minimal'' subset of swaps so that any swap follows from the swaps in the subset.

\paragraph{Formal explanation:} The algebra $\B$ is defined using generators and relations as follows. The generators are the letters $L$ and $R$. A \textit{word} is a concatenation of letters, and a \textit{balanced} word consists of equal numbers of both letters. The defining relations of $\B$ are that any two balanced words commute. It turns out that many of these relations are redundant. The main goal of this article is to find a minimal set of relations for $\B$; more precisely, a minimal subset of the original set of relations that can be used as the defining relations of $\B$. It will be seen that this minimal subset is not unique; the main result gives a family of minimal subsets. One of the subsets in the family is then chosen to be studied in detail.

\paragraph{Motivation:} The balanced algebra $\B$ comes up in algebraic graph theory in the following way. Start with a graph $\Gamma$, and choose a vertex $\alpha$ as a base vertex. The vertex set of $\Gamma$ is partitioned into sets called \textit{subconstituents}; the $i\ith$ subconstituent consists of the vertices at distance $i$ from $\alpha$. The vertices of $\Gamma$ form a basis of a vector space called the \textit{standard module}. The \textit{raising} matrix $R$ and the \textit{lowering} matrix $L$ act on this basis by sending a vertex in the $i\ith$ subconstituent to the sum of its neighbors in the $(i+1)\st$ or $(i-1)\st$ subconstituent, respectively.

Under some assumptions ($\Gamma$ is distance-regular and bipartite), the matrices $L$ and $R$, together with certain projection matrices, generate an algebra called the \textit{subconstituent algebra} $T$ of $\Gamma$ with respect to $\alpha$. Certain well-behaved irreducible $T$-modules are called \textit{thin} modules. Under the assumptions mentioned above, it is known that the balanced words in $L$ and $R$ commute if and only if every irreducible $T$-module is thin \cite{paul}. In this case the graph $\Gamma$ is called thin with respect to $\alpha$. Studying the balanced algebra may help to better understand thin graphs.

\paragraph{Organization of the article:} In Section 2, it is explained how the condition ``balanced words commute'' comes up in algebraic graph theory. The balanced algebra $\B$ is defined in Section 3. The concept of swaps, used in most proofs, is also explained in that section. Section 4 is for introducing some useful tools. A family of minimal sets of relations for $\B$ is found in Section 5. In Section 6, one member of the family is studied in detail.

\pagebreak
\section{Motivation}

This section gives a bit more detail about how the balanced algebra comes up in algebraic graph theory, in the study of distance-regular graphs. The familiar cube is an example of a distance-regular graph, and it is used as a running example to illustrate the concepts discussed. Reading this section is not necessary for understanding the rest of the article.

\paragraph{Assumptions regarding graphs:} Throughout this section, $\Gamma$ denotes a finite, undirected, and connected graph without any loops or repeated edges. The vertex set of $\Gamma$ is denoted by $\X$, and the number of vertices by $n$.

\paragraph{Definition (distance and diameter):} For a nonnegative integer $k$, a \textit{path} of length $k$ in $\Gamma$ is a sequence $x_0,x_1,\dots, x_k$ of distinct vertices such that for $1\leq i\leq k$, the vertices $x_{i-1}$ and $x_i$ are adjacent. This path is said to be \textit{from} $x_0$ \textit{to} $x_k$. The \textit{path-length distance function} $\partial$ is defined as follows: for vertices $x$ and $y$ of $\Gamma$, their distance $\partial(x,y)$ is the minimal length of a path from $x$ to $y$. The \textit{diameter} $d=d(\Gamma)$ is defined to be the maximal distance between two vertices of $\Gamma$.

%\pagebreak

\paragraph{Definition (distance-regular graph and intersection numbers):} The graph $\Gamma$ is called \textit{distance-regular} if, for $0\leq i,j\leq d$, the size of the set $\{z\colon \partial(x,z)=i,\partial(z,y)=j\}$ does not depend on the vertices $x$ and $y$, but only on their distance $h=\partial(x,y)$. The size of the above set is denoted by $p_{ij}^h$. The numbers $p_{ij}^h$ ($0\leq h,i,j\leq d$) are called the \textit{intersection numbers} of $\Gamma$.

\paragraph{Definition (bipartite graph):} The graph $\Gamma$ is \textit{bipartite} if its vertex set $\X$ can be partitioned into two subsets with the property that two vertices belonging to the same subset are never adjacent.

\paragraph{More assumptions regarding graphs:} For the rest of the section, it is assumed that $\Gamma$ is distance-regular and bipartite. A vertex $\alpha$ of $\Gamma$ is chosen as a base vertex.

\paragraph{General notes about intersection numbers:} Firstly, the intersection numbers are symmetric in the sense that $p_{ij}^h=p_{ji}^h$. Secondly, the distance function $\partial$ satisfies the triangle inequality, which means that $p_{ij}^h=0$ if the sum of two of the numbers $h,i,j$ is less than the third one. Thirdly, $p_{ij}^h=0$ if $h+i+j$ is odd. (This is because bipartite graphs do not have odd cycles.)

\paragraph{Example:} The cube graph $Q_3$ has vertex set $\{0,1\}^3$, and two vertices are adjacent if they differ in exactly one coordinate. Note that the distance between two vertices is equal to the number of coordinates at which they differ. As there are three coordinates, the diameter of $Q_3$ is 3.

\begin{figure}[H]
\begin{center}
\scalebox{0.8}{
\begin{tikzpicture}
    \newcommand\h{1.4}
    \node[circle,fill=black,scale=0.4] at (0,0) (a) {};
    \node[circle,scale=0.4] at (-2,\h) (b1) {};
    \node[circle,scale=0.4] at (0,2) (b2) {};
    \node[circle,scale=0.4] at (2,\h) (b3) {};
    \node[circle,fill=black,scale=0.4] at (-2,2+\h) (c1) {};
    \node[circle,fill=black,scale=0.4] at (0,2*\h) (c2) {};
    \node[circle,fill=black,scale=0.4] at (2,2+\h) (c3) {};
    \node[circle,scale=0.4] at (0,2+2*\h) (d) {};
    \draw (b1)--(a)--(b3);
    \draw (a)--(b2);
    \draw (b1)--(c2)--(b3);
    \draw (c1)--(b2)--(c3);
    \draw (b1)--(c1)--(d)--(c3)--(b3);
    \draw (c2)--(d);
    \draw [semithick] (-2,\h) circle (1.9pt);
    \draw [semithick] (0,2) circle (1.9pt);
    \draw [semithick] (2,\h) circle (1.9pt);
    \draw [semithick] (0,2+2*\h) circle (1.9pt);
    \node[below = .5mm of a] {$(0,0,0)$};
    \node[left = .5mm of b1] {$(0,1,0)$};
    \node[right = .5mm of b3] {$(1,0,0)$};
    %\node[below right = .3mm of b2] {$(0,0,1)$};
    \node at (-3/5,8/5) {$(0,0,1)$};
    \node[left = .5mm of c1] {$(0,1,1)$};
    \node[right = .5mm of c3] {$(1,0,1)$};
    \node at (3/5,2*\h+2/5) {$(1,1,0)$};
    %\node[above left = .3mm of c2] {$(1,1,0)$};
    \node[above = .5mm of d]{$(1,1,1)$};
\end{tikzpicture}
}
\end{center}
    \caption{The cube graph $Q_3$.}
    \label{figuresimplecube}
\end{figure}
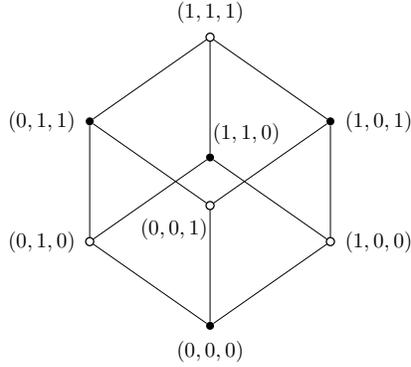

The graph $Q_3$ is distance-regular. The table below shows the intersection numbers $p_{ij}^h$ for the triples $(h,i,j)$ which satisfy the triangle inequality and for which $h+i+j$ is even; symmetry of intersection numbers allows to save space by listing only values with $i\geq j$.

\begin{center}
\begin{tabular}{|*{16}{c|}}
    \hline
    $h$ & 0&0&0&0&1&1&1&2&2&2&2&2&3&3&3\\
    \hline
    $i$ & 1&2&3&0&1&2&3&1&2&2&3&3&2&3&3\\   
    \hline
    $j$ & 1&2&3&0&0&1&2&1&0&2&1&3&1&0&2\\
    \hline
    $p_{ij}^h$ & 3&3&1&1&1&2&1&2&1&2&1&0&3&1&0\\
    \hline
\end{tabular}
\end{center}

\paragraph{Definition (standard module and subconstituents):} The \textit{standard module} $V$ of $\Gamma$ is a $\C$-vector space with basis $\{v\colon v\in \X\}$. For $0\leq i\leq d$, the set $\Gamma_i(\alpha)=\{z\in \X \colon \partial(x,z)=i\}$ is called the $i\ith$ \textit{subconstituent} of $\Gamma$ (with respect to $\alpha$). Let $\text{Mat}_\X(\C)$ denote the algebra consisting of square matrices over $\C$ with rows and columns indexed by $\X$. The algebra $\text{Mat}_\X(\C)$ acts on $V$ by left multiplication.

\paragraph{Example:} For the graph $Q_3$, the vertex $\alpha=(0,0,0)$ is chosen as the base vertex. The standard module has dimension 8. See Figure \ref{figurecube} for an illustration of $Q_3$ and its subconstituents.

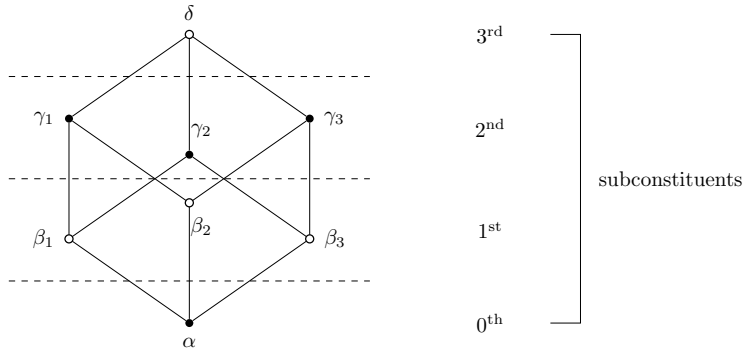
\begin{figure}[H]
\begin{center}
\scalebox{0.8}{
\begin{tikzpicture}
    \newcommand\h{1.4}
    \node[circle,fill=black,scale=0.4] at (0,0) (a) {};
    \node[circle,scale=0.4] at (-2,\h) (b1) {};
    \node[circle,scale=0.4] at (0,2) (b2) {};
    \node[circle,scale=0.4] at (2,\h) (b3) {};
    \node[circle,fill=black,scale=0.4] at (-2,2+\h) (c1) {};
    \node[circle,fill=black,scale=0.4] at (0,2*\h) (c2) {};
    \node[circle,fill=black,scale=0.4] at (2,2+\h) (c3) {};
    \node[circle,scale=0.4] at (0,2+2*\h) (d) {};
    \draw (b1)--(a)--(b3);
    \draw (a)--(b2);
    \draw (b1)--(c2)--(b3);
    \draw (c1)--(b2)--(c3);
    \draw (b1)--(c1)--(d)--(c3)--(b3);
    \draw (c2)--(d);
    \draw [semithick] (-2,\h) circle (1.9pt);
    \draw [semithick] (0,2) circle (1.9pt);
    \draw [semithick] (2,\h) circle (1.9pt);
    \draw [semithick] (0,2+2*\h) circle (1.9pt);
    \node[below = .5mm of a] {$\alpha$};
    \node[left = .5mm of b1] {$\beta_1$};
    \node[right = .5mm of b3] {$\beta_3$};
    \node at (1/5,8/5) {$\beta_2$};
    %\node[below = .3mm of b2] {$b_2=(0,0,1)$};
    \node[left = .5mm of c1] {$\gamma_1$};
    \node[right = .5mm of c3] {$\gamma_3$};
    \node at (1/5,2*\h+2/5) {$\gamma_2$};
    %\node[above = .3mm of c2] {$c_2=(1,1,0)$};
    \node[above = .5mm of d]{$\delta$};
    \draw[dashed] (-3,\h/2)--(3,\h/2);
    \draw[dashed] (-3,\h+1)--(3,\h+1);
    \draw[dashed] (-3, 2+3*\h/2)--(3,2+3*\h/2);
    \node at (5,0) {$0\ith$};
    \node at (5,3*\h/4+1/2) {$1^\text{st}$};
    \node at (5,5*\h/4+3/2) {$2^\text{nd}$};
    \node at (5,2+2*\h) {$3^\text{rd}$};
    \draw (6,0)--(6.5,0)--(6.5,2+2*\h)--(6,2+2*\h);
    \node at (8,1+\h) {subconstituents};
\end{tikzpicture}
}
\end{center}
    \caption{The graph $Q_3$ with base vertex $\alpha$. The dashed lines are used to separate the subconstituents.}
    \label{figurecube}
\end{figure}

\paragraph{Definition (raising and lowering matrices):} The raising matrix $R$ and the lowering matrix $L$ are matrices in $\text{Mat}_\X(\C)$ which act on $V$ as follows. Let $0\leq i\leq d$, and let $v$ be a vertex in the $i\ith$ subconstituent. Then $Rv$ is the sum of the neighbors of $v$ in the $(i+1)\st$ subconstituent (or 0 if $i=d$). Similarly, $Lv$ is the sum of the neighbors of $v$ in the $(i-1)\st$ subconstituent (or 0 if $i=0$).

\paragraph{Example:} The table below shows how the raising and lowering matrices act on the vertices of $Q_3$. The vertex labeling is from Figure \ref{figurecube}.

\begin{center}
\begin{tabular}{|*{3}{c|}}
    \hline
    $v$ & $Rv$ & $Lv$\\
    \hline
    $\alpha$ & $\beta_1+\beta_2+\beta_3$ & 0\\
    \hline
    $\beta_1$ & $\gamma_1+\gamma_2$ & $\alpha$\\
    \hline
    $\beta_2$ & $\gamma_1+\gamma_3$ & $\alpha$\\
    \hline
    $\beta_3$ & $\gamma_2+\gamma_3$ & $\alpha$\\
    \hline
    $\gamma_1$ & $\delta$ & $\beta_1+\beta_2$\\
    \hline
    $\gamma_2$ & $\delta$ & $\beta_1+\beta_3$\\
    \hline
    $\gamma_3$ & $\delta$ & $\beta_2+\beta_3$\\
    \hline
    $\delta$ & 0 & $\gamma_1+\gamma_2+\gamma_3$\\
    \hline
\end{tabular}
\end{center}

\paragraph{Note:} When explicitly writing down matrices associated with $Q_3$, the rows and columns will be indexed in the order given by the leftmost column in the table above.

\paragraph{Definition (adjacency matrix):} The adjacency matrix $A$ of $\Gamma$ is a matrix in $\text{Mat}_\X(\C)$ which acts on $V$ by sending a vertex to the sum of its neighbors.

\paragraph{A matrix equation:} Keeping in mind that $\Gamma$ is bipartite (which implies that two vertices inside the same subconstituent cannot be adjacent), the neighbors of a vertex in the $i\ith$ subconstituent can only be in the $(i+1)\st$ or $(i-1)\st$ subconstituent. This implies that the raising, lowering, and adjacency matrices are related via the equation $A=R+L$.

\paragraph{Example:} For $Q_3$, the matrix equation $A=R+L$ looks as follows.

%\vspace{-0.5cm}

\begin{align*}
\overset{\raise.5em \hbox{A}}{
\begin{bsmallmatrix}
0 & \bf{1} & \bf{1} & \bf{1} & 0 & 0 & 0 & 0 \\
\bf{1} & 0 & 0 & 0 & \bf{1} & \bf{1} & 0 & 0 \\
\bf{1} & 0 & 0 & 0 & \bf{1} & 0 & \bf{1} & 0 \\
\bf{1} & 0 & 0 & 0 & 0 & \bf{1} & \bf{1} & 0 \\
0 & \bf{1} & \bf{1} & 0 & 0 & 0 & 0 & \bf{1} \\
0 & \bf{1} & 0 & \bf{1} & 0 & 0 & 0 & \bf{1} \\
0 & 0 & \bf{1} & \bf{1} & 0 & 0 & 0 & \bf{1} \\
0 & 0 & 0 & 0 & \bf{1} & \bf{1} & \bf{1} & 0
\end{bsmallmatrix}
}=
\overset{\raise.5em \hbox{R}}{
\begin{bsmallmatrix}
0 & 0 & 0 & 0 & 0 & 0 & 0 & 0 \\
\bf{1} & 0 & 0 & 0 & 0 & 0 & 0 & 0 \\
\bf{1} & 0 & 0 & 0 & 0 & 0 & 0 & 0 \\
\bf{1} & 0 & 0 & 0 & 0 & 0 & 0 & 0 \\
0 & \bf{1} & \bf{1} & 0 & 0 & 0 & 0 & 0 \\
0 & \bf{1} & 0 & \bf{1} & 0 & 0 & 0 & 0 \\
0 & 0 & \bf{1} & \bf{1} & 0 & 0 & 0 & 0 \\
0 & 0 & 0 & 0 & \bf{1} & \bf{1} & \bf{1} & 0
\end{bsmallmatrix}
}+
\overset{\raise.5em \hbox{L}}{
\begin{bsmallmatrix}
0 & \bf{1} & \bf{1} & \bf{1} & 0 & 0 & 0 & 0 \\
0 & 0 & 0 & 0 & \bf{1} & \bf{1} & 0 & 0 \\
0 & 0 & 0 & 0 & \bf{1} & 0 & \bf{1} & 0 \\
0 & 0 & 0 & 0 & 0 & \bf{1} & \bf{1} & 0 \\
0 & 0 & 0 & 0 & 0 & 0 & 0 & \bf{1} \\
0 & 0 & 0 & 0 & 0 & 0 & 0 & \bf{1} \\
0 & 0 & 0 & 0 & 0 & 0 & 0 & \bf{1} \\
0 & 0 & 0 & 0 & 0 & 0 & 0 & 0
\end{bsmallmatrix}.
}
\end{align*}

\paragraph{Projections:} For $0\leq i\leq d$, let $E_i^\star$ denote the matrix in $\text{Mat}_\X(\C)$ which acts on $V$ as follows: A vertex in the $i\ith$ subconstituent is sent to itself, and vertices in all other subconstituents are sent to zero. Note that $I=E_0^\star+E_1^\star+\dots+E_d^\star$.

\paragraph{Example:} The diameter of $Q_3$ is 3 (as seen earlier), and the equation $I=E_0^\star+E_1^\star+E_2^\star+E_3^\star$ looks as follows.

\begin{align*}
\overset{\raise.5em \hbox{$I$}}{
\begin{bsmallmatrix}
\bf{1} & 0 & 0 & 0 & 0 & 0 & 0 & 0 \\
0 & \bf{1} & 0 & 0 & 0 & 0 & 0 & 0 \\
0 & 0 & \bf{1} & 0 & 0 & 0 & 0 & 0 \\
0 & 0 & 0 & \bf{1} & 0 & 0 & 0 & 0 \\
0 & 0 & 0 & 0 & \bf{1} & 0 & 0 & 0 \\
0 & 0 & 0 & 0 & 0 & \bf{1} & 0 & 0 \\
0 & 0 & 0 & 0 & 0 & 0 & \bf{1} & 0 \\
0 & 0 & 0 & 0 & 0 & 0 & 0 & \bf{1}
\end{bsmallmatrix}
}=
\overset{\raise.5em \hbox{$E_0^\star$}}{
\begin{bsmallmatrix}
\bf{1} & 0 & 0 & 0 & 0 & 0 & 0 & 0 \\
0 & 0 & 0 & 0 & 0 & 0 & 0 & 0 \\
0 & 0 & 0 & 0 & 0 & 0 & 0 & 0 \\
0 & 0 & 0 & 0 & 0 & 0 & 0 & 0 \\
0 & 0 & 0 & 0 & 0 & 0 & 0 & 0 \\
0 & 0 & 0 & 0 & 0 & 0 & 0 & 0 \\
0 & 0 & 0 & 0 & 0 & 0 & 0 & 0 \\
0 & 0 & 0 & 0 & 0 & 0 & 0 & 0
\end{bsmallmatrix}
}+
\overset{\raise.5em \hbox{$E_1^\star$}}{
\begin{bsmallmatrix}
0 & 0 & 0 & 0 & 0 & 0 & 0 & 0 \\
0 & \bf{1} & 0 & 0 & 0 & 0 & 0 & 0 \\
0 & 0 & \bf{1} & 0 & 0 & 0 & 0 & 0 \\
0 & 0 & 0 & \bf{1} & 0 & 0 & 0 & 0 \\
0 & 0 & 0 & 0 & 0 & 0 & 0 & 0 \\
0 & 0 & 0 & 0 & 0 & 0 & 0 & 0 \\
0 & 0 & 0 & 0 & 0 & 0 & 0 & 0 \\
0 & 0 & 0 & 0 & 0 & 0 & 0 & 0
\end{bsmallmatrix}
}+
\overset{\raise.5em \hbox{$E_2^\star$}}{
\begin{bsmallmatrix}
0 & 0 & 0 & 0 & 0 & 0 & 0 & 0 \\
0 & 0 & 0 & 0 & 0 & 0 & 0 & 0 \\
0 & 0 & 0 & 0 & 0 & 0 & 0 & 0 \\
0 & 0 & 0 & 0 & 0 & 0 & 0 & 0 \\
0 & 0 & 0 & 0 & \bf{1} & 0 & 0 & 0 \\
0 & 0 & 0 & 0 & 0 & \bf{1} & 0 & 0 \\
0 & 0 & 0 & 0 & 0 & 0 & \bf{1} & 0 \\
0 & 0 & 0 & 0 & 0 & 0 & 0 & 0
\end{bsmallmatrix}
}+
\overset{\raise.5em \hbox{$E_3^\star$}}{
\begin{bsmallmatrix}
0 & 0 & 0 & 0 & 0 & 0 & 0 & 0 \\
0 & 0 & 0 & 0 & 0 & 0 & 0 & 0 \\
0 & 0 & 0 & 0 & 0 & 0 & 0 & 0 \\
0 & 0 & 0 & 0 & 0 & 0 & 0 & 0 \\
0 & 0 & 0 & 0 & 0 & 0 & 0 & 0 \\
0 & 0 & 0 & 0 & 0 & 0 & 0 & 0 \\
0 & 0 & 0 & 0 & 0 & 0 & 0 & 0 \\
0 & 0 & 0 & 0 & 0 & 0 & 0 & \bf{1}
\end{bsmallmatrix}
}.
\end{align*}

\paragraph{Definition ($T$-modules):} Let $T$ be the subalgebra of $\text{Mat}_\X(\C)$ generated by $R$, $L$, and $E_0^\star,\dots, E_d^\star$. A \textit{$T$-module} is a subspace of $V$ that is closed under the action of $T$. A $T$-module $W$ is \textit{irreducible} if $W\neq 0$ and $W$ does not contain any $T$-modules other than 0 and $W$ itself. An irreducible $T$-module $W$ is \textit{thin} if $W\cap E_i^\star V$ has dimension 0 or 1 for $0\leq i\leq d$. The graph $\Gamma$ is called \textit{thin with respect to} $\alpha$ if every irreducible $T$-module is thin.

\paragraph{A key result:} The graph $\Gamma$ is thin with respect to $\alpha$ if and only if the balanced words in $R$ and $L$ commute.

\textit{Sketch of proof.} In \cite{paul}, it is shown that every irreducible $T$-module is thin if and only if  $E_i^\star TE_i^\star$ is commutative for $0\leq i\leq d$. The latter condition can be shown to be equivalent to balanced words in $R$ and $L$ commuting.

\paragraph{Further reading:} The graph $Q_3$ is a special case of a hypercube $Q_d$. The graph $Q_d$ is discussed in detail in \cite{junie}.

\pagebreak

\section{Ideals and swapping}

This section describes how the balanced algebra is obtained as the quotient by an ideal $\J$ of the free algebra with two generators $L$ and $R$. (This is exactly the generators and relations approach from Section 1, done in detail.) The main goal of this article is to find a minimal generating set for $\J$. A key result of this section gives a powerful perspective (``swaps'') for looking at ideal membership.

\paragraph{Definition (words and free algebra):} Let $\A$ be the free $\C$-algebra with the generators $L$ and $R$. The two generators are called \textit{letters}, and a \textit{word of length $n$} is a product (concatenation) $a_1a_2\dots a_n$, where $a_i$ is a letter for $1\leq i\leq n$. A \textit{subword} of $a_1a_2\dots a_n$ is a word $a_ka_{k+1}\dots a_l$ where $1\leq k\leq l\leq n$. The length of a word $W$ is denoted by $l(W)$. The empty word has length 0, and it is the multiplicative identity of $\A$. All other words have positive length and are called nonempty. As a complex vector space, $\A$ has a basis consisting of all possible words in the two letters $L$ and $R$.

\paragraph{Example:} $LLL$ and $RL$ are words, and $2LLL+5RL$ is an element of $\A$. Some subwords of $RRLRLL$ are $RRLR$ and $LRL$.

\paragraph{Important note:} In this article, a ``word'' always refers to a word in $\A$, and ``an ideal of $\A$'' is used to mean a two-sided ideal of $\A$.

\paragraph{Definition (balanced words and balanced algebra):} A word is called \textit{balanced} if the letters $L$ and $R$ appear equally many times in it. Define the set
$$S=\{FG-GF\colon\; F\text{ and }G\text{ are nonempty balanced words}\},$$
and let $\J$ be the ideal of $\A$ generated by $S$. The quotient algebra $\B=\A/\J$ is called the \textit{balanced algebra}.

\paragraph{Example:} The words $LR$ and $RRLRLL$ are both nonempty balanced words, which implies that $(LR)(RRLRLL)-(RRLRLL)(LR)\in S$. See Figure \ref{figureword} for an illustration of the word $RRLRLL$.

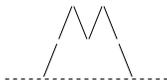
\begin{figure}[H]
\begin{center}
\scalebox{0.5}{
\begin{tikzpicture}
    \newcommand\s{0.4}
    \node[circle,scale=0.4] at (0,0) (0) {};
    \node[circle,scale=0.4] at (1*\s,1) (1) {};
    \node[circle,scale=0.4] at (2*\s,2) (2) {};
    \node[circle,scale=0.4] at (3*\s,1) (3) {};
    \node[circle,scale=0.4] at (4*\s,2) (4) {};
    \node[circle,scale=0.4] at (5*\s,1) (5) {};
    \node[circle,scale=0.4] at (6*\s,0) (6) {};
    \draw (0)--(1)--(2)--(3)--(4)--(5)--(6);
    \draw[dashed] (-1,0)--(6*\s+1,0);
    %\node[] at (3*\s, -1) {\LARGE $RRLRLL$};
\end{tikzpicture}
}
\end{center}
    \caption{The word $RRLRLL$. An ascending line segment represents the letter $R$ and a descending line segment represents the letter $L$. (Words are drawn from left to right.) As the word is balanced, the word begins and ends at the same vertical level; in this picture, a dashed line is drawn at that level. }
    \label{figureword}
\end{figure}

\paragraph{Goal:} The main goal of this article is to find a minimal subset of $S$ that generates $\J$. Note that there are many possible such subsets; in fact, the main result of this article (Theorem \ref{minimal}) gives an infinite family of them. On the way, some choices are made, so there might well exist some ``nice'' minimal subset of $S$ that generates $\J$ but which does not belong in this family.

\paragraph{Definition (equivalence of words):} Define the binary relation $\sim$ on the set of words as follows: if $X$ and $Y$ are words, then $X\sim Y$ whenever $X-Y\in\J$. Note that $\sim$ is an equivalence relation. Whenever equivalence classes of words are mentioned, they refer to equivalence classes with respect to the relation $\sim$.

\paragraph{Example:} $LRRRLL-RRLLLR=(LR)(RRLL)-(RRLL)(LR)\in S\subset \J$, and this implies that $LRRRLL\sim RRLLLR$.

\paragraph{Note:} The equivalence relation defined above is the same one that would normally be used for determining whether two elements of $\A$ correspond to the same element in the quotient $\B$; the only difference is that this equivalence relation is only used for words (and not linear combinations of words). Going forward, elements of $\B$ are not discussed, and instead everything is done in $\A$.

\paragraph{Definition (swaps):} Consider two nonempty balanced words $F$ and $G$, and assume that they appear next to each other inside a word $W$. Then there exist words $W_1$ and $W_2$ so that $W=W_1FGW_2$, or $W=W_1GFW_2$. Switching the places of $F$ and $G$ is called a \textit{swap} of type $(F,G)$. The two words $W_1FGW_2$ and $W_1GFW_2$ are said to be \textit{related by a swap}, or more precisely, related by a swap of type $(F,G)$. Note that a swap of type $(F,G)$ is the same thing as a swap of type $(G,F)$.

\paragraph{Example:} The words $RLLR=(RL)(LR)$ and $LRRL=(LR)(RL)$ are related by a swap of type $(RL,LR)$.

\paragraph{Note:} Sometimes two words can be related by a swap in multiple ways, as seen in the following example.

\paragraph{Example:} The words $RRLLRL=R(RL)(LR)L$ and $RLRRLL=R(LR)(RL)L$ are related by a swap of type $(RL,LR)$. By rearranging the parentheses, the words can be written as $(RRLL)(RL)$ and $(RL)(RRLL)$, so the words are also related by a swap of type $(RRLL,RL)$. Figure \ref{figureswap} illustrates the two words.

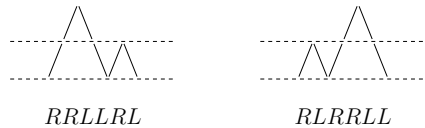
\begin{figure}[H]
\begin{center}
\scalebox{0.5}{
\begin{tikzpicture}
    \newcommand\s{0.4}
    \node[circle,scale=0.4] at (0,0) (0) {};
    \node[circle,scale=0.4] at (1*\s,1) (1) {};
    \node[circle,scale=0.4] at (2*\s,2) (2) {};
    \node[circle,scale=0.4] at (3*\s,1) (3) {};
    \node[circle,scale=0.4] at (4*\s,0) (4) {};
    \node[circle,scale=0.4] at (5*\s,1) (5) {};
    \node[circle,scale=0.4] at (6*\s,0) (6) {};
    \draw (0)--(1)--(2)--(3)--(4)--(5)--(6);
    \draw[dashed] (-1,0)--(6*\s+1,0);
    \draw[dashed] (-1,1)--(6*\s+1,1);
    \node[] at (3*\s, -1) {\LARGE $RRLLRL$};
\end{tikzpicture}
    \hspace{2cm}
\begin{tikzpicture}
    \newcommand\s{0.4}
    \node[circle,scale=0.4] at (0,0) (0) {};
    \node[circle,scale=0.4] at (1*\s,1) (1) {};
    \node[circle,scale=0.4] at (2*\s,0) (2) {};
    \node[circle,scale=0.4] at (3*\s,1) (3) {};
    \node[circle,scale=0.4] at (4*\s,2) (4) {};
    \node[circle,scale=0.4] at (5*\s,1) (5) {};
    \node[circle,scale=0.4] at (6*\s,0) (6) {};
    \draw (0)--(1)--(2)--(3)--(4)--(5)--(6);
    \draw[dashed] (-1,0)--(6*\s+1,0);
    \draw[dashed] (-1,1)--(6*\s+1,1);
    \node[] at (3*\s, -1) {\LARGE $RLRRLL$};
\end{tikzpicture}
}
\end{center}
    \caption{The words $RRLLRL$ and $RLRRLL$ can be obtained from each other by switching the places of the ``peak'' $RL$ and ``valley'' $LR$. The higher dashed line indicates the level where the product of these words starts and ends. Alternatively, they can be obtained from each other by switching the places of the ``high peak'' $RRLL$ and the ``low peak'' $RL$. The lower dashed line indicates the level where the product of these words starts and ends.}
    \label{figureswap}
\end{figure}

\pagebreak

\paragraph{Definition (sequence of swaps):} Let $X$ and $Y$ be words. Assume that $Z_1,Z_2,\dots, Z_k$ are words with $Z_1=X$ and $Z_k=Y$. If $Z_i$ and $Z_{i+1}$ are related by a swap for $1\leq i \leq k-1$, then the words $X$ and $Y$ are said to have a \textit{sequence of swaps} between them. (Note that the situation is symmetric in the sense that if $X$ and $Y$ have a sequence of swaps between them, then so do $Y$ and $X$.)

\white
\begin{lemma}
\label{idealswap}
Let $I$ be any set, and let $F_i$ and $G_i$ be balanced words for all $i\in I$. Let $X$ and $Y$ be any two words. The following are equivalent:
\begin{enumerate}[label = (\roman*)]
    \item $X-Y$ is in the ideal $\mathcal{K}$ generated by $\{F_iG_i-G_iF_i\}_{i\in I}$;
    \item there is a sequence of swaps between the words $X$ and $Y$, where every swap is of type $(F_i,G_i)$ for some $i\in I$.
\end{enumerate}
\end{lemma}

\begin{proof}
(i) $\implies$ (ii): The assumption $X-Y\in \mathcal{K}$ implies that $X-Y=\sum_{j=1}^m \alpha_j W_{j,1}(F_jG_j-G_jF_j)W_{j,2}$, where $W_{j,1}$ and $W_{j,2}$ are words and $\alpha_j\in \C$ are nonzero scalars; for each $j=1,\dots, m$ there exists $i\in I$ so that $(F_j,G_j)=(F_i,G_i)$. The term $W_{j,1}(F_jG_j-G_jF_j)W_{j,2}$ involves the words $X_j=W_{j,1}F_jG_jW_{j,2}$ and $Y_j=W_{j,1}G_jF_jW_{j,2}$, which are related by a swap of type $(F_j,G_j)$. With the simplified notation,
\begin{equation} \label{xysum}
    X-Y=\sum_{j=1}^m \alpha_j (X_j- Y_j).
\end{equation}

Using the equation (\ref{xysum}), define a graph $\Gamma$ as follows: the vertices of $\Gamma$ are all the words appearing in the terms on the right hand side, that is, the vertex set is $\{W\colon W=X_j \text{ or } W=Y_j \text{ for some } j\}$. Two vertices are adjacent if there is an index $j$ so that one of the words is equal to $X_j$ and the other one is $Y_j$. Note that the equation (\ref{xysum}), together with the fact that distinct words in $\A$ are linearly independent, implies that both $X$ and $Y$ are vertices.

If it can be shown that $X$ and $Y$ are in the same connected component, then there is a sequence of swaps between $X$ and $Y$. This is because each edge comes from a pair of words related by a swap, as mentioned above. For the argument, it is convenient to make $\Gamma$ into a weighted graph. Impose a weighting on the vertices of $\Gamma$ as follows: for a vertex $W$, the weight of $W$ is the coefficient of $W$ on the right hand side of (\ref{xysum}), when the sum is distributed. Because of that same equation, the vertex $X$ has weight 1, the vertex $Y$ has weight $-1$, and all other vertices have weight 0.

For each $j$, the vertices $X_j$ and $Y_j$ are in the same connected component, by the definition of $\Gamma$. Therefore the sum on the right hand side of (\ref{xysum}) can be separated into sums over each connected component, and this implies that the sum of weights over any connected component is zero. Now, if $Y$ is not in the connected component of $X$, then the sum of weights of this component is 1, which is a contradiction. Therefore $X$ and $Y$ are in the same component and so there is a sequence of swaps between $X$ and $Y$.

(ii) $\implies$ (i): It needs to be shown that $X-Y\in \mathcal{K}$, and this will be done by induction on the number of swaps in the sequence of swaps between $X$ and $Y$. First assume that the sequence consists of a single swap of type $(F,G)$, where $F=F_i$ and $G=G_i$ for some $i$. This means that there are (possibly empty) words $W_1$ and $W_2$ so that $X=W_1FGW_2$ and $Y=W_1GFW_2$. Then $X-Y=W_1(FG-GF)W_2\in \mathcal{K}$.

\pagebreak

Now assume that there is a sequence of $n$ swaps between $X$ and $Y$, with $n\geq 2$. Then there exists a word $W$ so that $X$ and $W$ have a sequence of $n-1$ swaps between them and $W$ and $Y$ are related by a single swap. By induction, both $X-W$ and $W-Y$ are in $\mathcal{K}$, and therefore so is $X-Y=(X-W)+(W-Y)$.
\end{proof}

\paragraph{Note:} The following proposition provides a powerful characterization for the equivalence of words, and it will be used in proofs throughout the article.

\white
\begin{proposition}
\label{equivalence}
Let $X$ and $Y$ be two words. Then $X\sim Y$ if and only if there is a sequence of swaps between $X$ and $Y$.
\end{proposition}
\begin{proof}
This is a special case of Lemma \ref{idealswap}, where the generating set is $S$.
\end{proof}

\white
\begin{corollary}
\label{finite}
Let $\mathcal{E}$ denote an equivalence class of words. Then
\begin{enumerate}[label = (\roman*)]
    \item $\mathcal{E}$ consists of words of equal length, and
    \item $\mathcal{E}$ is finite.
\end{enumerate}
\end{corollary}
\begin{proof}
Follows from Proposition \ref{equivalence}.
\end{proof}

\pagebreak
\section{Prime words and elevation}

This section introduces the concepts of prime words and elevation. One of the results in this section states that if the balanced words in the definition of $S$ are replaced with prime words, then the resulting set $S'$ generates $\J$. The set $S'$ is used as an intermediate step in finding a minimal generating set for $\J$.

\paragraph{Definition (prime words):} A word is called \textit{prime} if it is nonempty, balanced, and cannot be written as the product of two nonempty balanced words.

\white
\begin{lemma}
\label{primebeginning}
A nonempty balanced word $a_1a_2\dots a_n$ is prime if and only if the word $a_1a_2\dots a_k$ is not balanced for $k=1,\dots n-1$.
\end{lemma}
\begin{proof}
Assume first that $a_1a_2\dots a_n$ is prime. If $a_1a_2\dots a_k$ is balanced for some $k$ with $1\leq k\leq n-1$, then the word $a_{k+1}\dots a_n$ is also balanced. This implies that $a_1a_2\dots a_n$ can be written as a product of two nonempty balanced words, which is a contradiction. For the converse direction, assume that the word $a_1a_2\dots a_k$ is not balanced for $k=1,\dots n-1$. Then $a_1a_2\dots a_n$ cannot be written as the product of two nonempty balanced words, so it is prime.
\end{proof}

\paragraph{Example:} The words $RL$, $LLRR$, and $RRRLRLLL$ are prime, which can be checked using Lemma \ref{primebeginning}. On the other hand, the word $RRLLRLRLLLRLRR$ is not prime, as it can be written as the product of $RRLL$, $RL$, $RL$, and $LLRLRR$, which are all primes. See Figure \ref{figureprimefactorization} for an illustration.

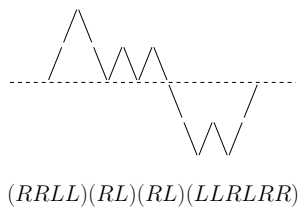
\begin{figure}[H]
\begin{center}
\scalebox{0.5}{
\begin{tikzpicture}
    \newcommand\s{0.4}
    \node[circle,scale=0.4] at (0,0) (0) {};
    \node[circle,scale=0.4] at (1*\s,1) (1) {};
    \node[circle,scale=0.4] at (2*\s,2) (2) {};
    \node[circle,scale=0.4] at (3*\s,1) (3) {};
    \node[circle,scale=0.4] at (4*\s,0) (4) {};
    \node[circle,scale=0.4] at (5*\s,1) (5) {};
    \node[circle,scale=0.4] at (6*\s,0) (6) {};
    \node[circle,scale=0.4] at (7*\s,1) (7) {};
    \node[circle,scale=0.4] at (8*\s,0) (8) {};
    \node[circle,scale=0.4] at (9*\s,-1) (9) {};
    \node[circle,scale=0.4] at (10*\s,-2) (10) {};
    \node[circle,scale=0.4] at (11*\s,-1) (11) {};
    \node[circle,scale=0.4] at (12*\s,-2) (12) {};
    \node[circle,scale=0.4] at (13*\s,-1) (13) {};
    \node[circle,scale=0.4] at (14*\s,0) (14) {};
    \draw (0)--(1)--(2)--(3)--(4)--(5)--(6)--(7)--(8)--(9)--(10)--(11)--(12)--(13)--(14);
    \draw[dashed] (-1,0)--(14*\s+1,0);
    \node[] at (7*\s, -3) {\LARGE $(RRLL)(RL)(RL)(LLRLRR)$};
\end{tikzpicture}
}
\end{center}
    \caption{Illustration of the word $RRLLRLRLLLRLRR=(RRLL)(RL)(RL)(LLRLRR)$. Each prime in the product starts and ends at the dotted line.}
    \label{figureprimefactorization}
\end{figure}

\white
\begin{lemma}
\label{primelemma}
A nonempty balanced word $W$ can be written uniquely as $W=P_1P_2\cdots P_k$, where $k$ is a positive integer and $P_i$ is a prime word for $1\leq i\leq k$.
\end{lemma}
\begin{proof}
Clear.
\end{proof}

\paragraph{Definition (prime factors):} Referring to Lemma \ref{primelemma}, the prime words $P_i$ are called \textit{prime factors} of $W$, and the \textit{number of prime factors of} $W$ refers to the number $k$.

\paragraph{A new set:} Define the set
$$S'=\{PQ-QP\colon\; P\text{ and }Q\text{ are prime words}\}.$$
Note that $S'\subset S$.

\white
\begin{lemma}
\label{sprime}
The set $S'$ generates the ideal $\J$.
\end{lemma}

\begin{proof}
Let $\J'$ denote the ideal generated by $S'$; the goal is to prove that $\J=\J'$. The inclusion $\J'\subset \J$ follows from the fact that $S'\subset S$. For the reverse inclusion, it suffices to show that if $F$ and $G$ are two nonempty balanced words, then $FG-GF\in \J'$. Let $p$ denote the number of prime factors in $FG$. The proof is by induction on $p$. As both $F$ and $G$ are nonempty, they both have at least one prime factor, which means that $p\geq 2$. If $p=2$, then $F$ and $G$ are both prime words and thus $FG-GF\in S'\subset \J'$.

Assume that $p>2$. Then either $F$ or $G$ is not prime; without loss of generality, it may be assumed that $F$ is not prime. Thus $F$ can be written as $F_1F_2$ where both $F_1$ and $F_2$ are nonempty balanced words. Now
\begin{align*}
    FG-GF&=F_1F_2G-GF_1F_2\\
    &=F_1F_2G-F_1GF_2+F_1GF_2-GF_1F_2\\
    &=F_1(F_2G-GF_2)+(F_1G-GF_1)F_2.
\end{align*}
Each of $F_1$ and $F_2$ has fewer prime factors than $F$, so the number of prime factors in both $F_1G$ and $F_2G$ is less than $p$. By induction, $F_iG-GF_i\in \J'$ for $i=1,2$. Using this together with the calculation from above, it can be concluded that $FG-GF\in \J'$. This concludes the proof.
\end{proof}

\paragraph{Definition (elevation):} For a letter $a$, assign a weight $\overline{a}$ as follows: $\overline{a}=1$ if $a=R$ and $\overline{a}=-1$ if $a=L$. Let $W=a_1a_2\dots a_n$ be a balanced word, and let $0\leq k\leq n$. The $k\ith$ \textit{elevation of} $W$ is denoted by $e_k(W)$ and given by $e_k(W)=\sum_{i=1}^k \overline{a_i}$. The values $e_k(W)$ form the \textit{elevation sequence} $Q(W)=\{e_k(W)\}_{k=0}^n$. The underlying multiset of $Q(W)$ is called the \textit{elevation multiset of} $W$, and it is denoted by $E(W)$.

\paragraph{Note:} The elevation sequence of any balanced word begins and ends with a zero.

\paragraph{First example:} The balanced word $W=RRLL$ has elevation sequence $Q(W)=\{0,1,2,1,0\}$, and elevation multiset $E(W)=\{0^2,1^2,2\}$. The exponent indicates how many times a number appears in the multiset.

\paragraph{Second example:} The balanced word $W=RRRLLRLLLLRRRL$ has elevation sequence $$Q(W)=\{0,1,2,3,2,1,2,1,0,-1,-2,-1,0,1,0\}$$ and elevation multiset $E(W)=\{-2,(-1)^2, 0^4, 1^4, 2^3, 3\}$. See Figure \ref{figuresequence} for an illustration.

\white
\begin{figure}[H]
\begin{center}
\scalebox{0.5}{
\begin{tikzpicture}
    \newcommand\s{0.4}
    \node[circle,scale=0.4] at (0,0) (0) {};
    \node[circle,scale=0.4] at (1*\s,1) (1) {};
    \node[circle,scale=0.4] at (2*\s,2) (2) {};
    \node[circle,scale=0.4] at (3*\s,3) (3) {};
    \node[circle,scale=0.4] at (4*\s,2) (4) {};
    \node[circle,scale=0.4] at (5*\s,1) (5) {};
    \node[circle,scale=0.4] at (6*\s,2) (6) {};
    \node[circle,scale=0.4] at (7*\s,1) (7) {};
    \node[circle,scale=0.4] at (8*\s,0) (8) {};
    \node[circle,scale=0.4] at (9*\s,-1) (9) {};
    \node[circle,scale=0.4] at (10*\s,-2) (10) {};
    \node[circle,scale=0.4] at (11*\s,-1) (11) {};
    \node[circle,scale=0.4] at (12*\s,0) (12) {};
    \node[circle,scale=0.4] at (13*\s,1) (13) {};
    \node[circle,scale=0.4] at (14*\s,0) (14) {};
    \draw (0)--(1)--(2)--(3)--(4)--(5)--(6)--(7)--(8)--(9)--(10)--(11)--(12)--(13)--(14);
    \draw[dashed] (-1,-2)--(14*\s+1,-2);
    \draw[dashed] (-1,-1)--(14*\s+1,-1);
    \draw[dashed] (-1,0)--(14*\s+1,0);
    \draw[dashed] (-1,1)--(14*\s+1,1);
    \draw[dashed] (-1,2)--(14*\s+1,2);
    \draw[dashed] (-1,3)--(14*\s+1,3);
    \node[] at (-2,-2) {\LARGE $-2$};
    \node[] at (-2,-1) {\LARGE $-1$};
    \node[] at (-2,0) {\LARGE $0$};
    \node[] at (-2,1) {\LARGE $1$};
    \node[] at (-2,2) {\LARGE $2$};
    \node[] at (-2,3) {\LARGE $3$};
\end{tikzpicture}
}
\end{center}
\caption{Illustration of the word $W=RRRLLRLLLLRRRL$. The dashed lines indicate the different elevations.}
\label{figuresequence}
\end{figure}
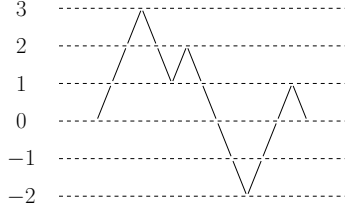

\white
\begin{lemma}
\label{equivmultiset}
If $X$ and $Y$ are balanced words with $X\sim Y$, then $E(X)=E(Y)$.
\end{lemma}
\begin{proof}
By Proposition \ref{equivalence}, there is a sequence of swaps between $X$ and $Y$. A swap does not change the elevation multiset but merely rearranges the entries in the elevation sequence. Therefore $X$ and $Y$ have the same elevation multiset.
\end{proof}

\paragraph{Example:} The words $RRRLLRLL$ and $RRLRRLLL$ are related by a swap of type $(RL,LR)$, so $RRRLLRLL\sim RRLRRLLL$. The table below shows the elevation sequences for both words. The values that switch places in the swap are underlined. The words share the same elevation multiset. See Figure \ref{figuremultiset} for illustration of the two words.

\begin{center}
\begin{tabular}{|c|c|c|}
    \hline
    $W$ & $RRRLLRLL$ & $RRLRRLLL$ \\
    \hline
    $Q(W)$ & $\{0,1,2,\underline{3},2,\underline{1},2,1,0\}$ & $\{0,1,2,\underline{1},2,\underline{3},2,1,0\}$\\
    \hline
    $E(W)$ & $\{0^2,1^3,2^3,3\}$ & $\{0^2,1^3,2^3,3\}$\\
    \hline
\end{tabular}
\end{center}

\begin{figure}[H]
\begin{center}
\scalebox{0.5}{
\begin{tikzpicture}
    \newcommand\s{0.4}
    \node[circle,scale=0.4] at (0,0) (0) {};
    \node[circle,scale=0.4] at (\s,1) (1) {};
    \node[circle,scale=0.4] at (2*\s,2) (2) {};
    \node[circle,scale=0.4] at (3*\s,3) (3) {};
    \node[circle,scale=0.4] at (4*\s,2) (4) {};
    \node[circle,scale=0.4] at (5*\s,1) (5) {};
    \node[circle,scale=0.4] at (6*\s,2) (6) {};
    \node[circle,scale=0.4] at (7*\s,1) (7) {};
    \node[circle,scale=0.4] at (8*\s,0) (8) {};
    \draw (0)--(1)--(2)--(3)--(4)--(5)--(6)--(7)--(8);
    \draw[dashed] (0,2)--(8*\s,2);
    \node[] at (4*\s, -1) {\LARGE $RR(RL)(LR)LL$};
\end{tikzpicture}
    \hspace{2cm}
\begin{tikzpicture}
    \newcommand\s{0.4}
    \node[circle,scale=0.4] at (0,0) (0) {};
    \node[circle,scale=0.4] at (\s,1) (1) {};
    \node[circle,scale=0.4] at (2*\s,2) (2) {};
    \node[circle,scale=0.4] at (3*\s,1) (3) {};
    \node[circle,scale=0.4] at (4*\s,2) (4) {};
    \node[circle,scale=0.4] at (5*\s,3) (5) {};
    \node[circle,scale=0.4] at (6*\s,2) (6) {};
    \node[circle,scale=0.4] at (7*\s,1) (7) {};
    \node[circle,scale=0.4] at (8*\s,0) (8) {};
    \draw (0)--(1)--(2)--(3)--(4)--(5)--(6)--(7)--(8);
    \draw[dashed] (0,2)--(8*\s,2);
    \node[] at (4*\s, -1) {\LARGE $RR(LR)(RL)LL$};
\end{tikzpicture}
}
\end{center}
    \caption{The two words are related by a swap of type $(RL,LR)$. The dashed line is at elevation 2, where the subword $(RL)(LR)$ (in the first word) and $(LR)(RL)$ (in the second word) begins and ends.}
    \label{figuremultiset}
\end{figure}
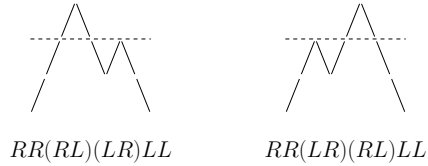

\white
\begin{proposition}
\label{reformulation}
Let $W$ be a balanced word with $l(W)\geq 2$. The following are equivalent:
\begin{enumerate}[label=(\roman*)]
    \item $W$ is prime;
    \item $e_k(W)\neq0$ for $1\leq k\leq l(W)-1$.
\end{enumerate}
\end{proposition}
\begin{proof}
This is a reformulation of Lemma \ref{primebeginning}.
\end{proof}

\white
\begin{corollary}
\label{samesign}
Let $P$ be a prime word. Then one of the following holds:
\begin{enumerate}[label=(\roman*)]
    \item $e_k(P)>0$ for $1\leq k\leq l(P)-1$;
    \item $e_k(P)<0$ for $1\leq k\leq l(P)-1$.
\end{enumerate}
\end{corollary}
\begin{proof}
Elevations are integers, and adjacent elevations $e_k(P)$ and $e_{k+1}(P)$ always differ by 1. If $e_k(P)<0$ and $e_m(P)>0$ for some $k,m$, then $e_l(P)=0$ for some $l$ between $k$ and $m$. This contradicts Proposition \ref{reformulation}.
\end{proof}

\paragraph{Definition (upper and lower primes):} Let $P$ be a prime word. Referring to Corollary \ref{samesign}, if $P$ satisfies (i), then $P$ is called an \textit{upper} prime. Similarly, if $P$ satisfies (ii), then $P$ is called a \textit{lower} prime. See Figure \ref{figureupperlower} for an illustrated example of upper and lower primes.

\begin{figure}[H]
\begin{center}
\scalebox{0.5}{
\begin{tikzpicture}
    \newcommand\s{0.4}
    \node[circle,scale=0.4] at (0,0) (0) {};
    \node[circle,scale=0.4] at (1*\s,1) (1) {};
    \node[circle,scale=0.4] at (2*\s,2) (2) {};
    \node[circle,scale=0.4] at (3*\s,1) (3) {};
    \node[circle,scale=0.4] at (4*\s,2) (4) {};
    \node[circle,scale=0.4] at (5*\s,3) (5) {};
    \node[circle,scale=0.4] at (6*\s,2) (6) {};
    \node[circle,scale=0.4] at (7*\s,3) (7) {};
    \node[circle,scale=0.4] at (8*\s,4) (8) {};
    \node[circle,scale=0.4] at (9*\s,3) (9) {};
    \node[circle,scale=0.4] at (10*\s,2) (10) {};
    \node[circle,scale=0.4] at (11*\s,1) (11) {};
    \node[circle,scale=0.4] at (12*\s,0) (12) {};
    \draw (0)--(1)--(2)--(3)--(4)--(5)--(6)--(7)--(8)--(9)--(10)--(11)--(12);
    \node[] at (6*\s, -1) {\LARGE $RRLRRLRRLLLL$};
\end{tikzpicture}
\hspace{2cm}
\begin{tikzpicture}
    \newcommand\s{0.4}
    \node[circle,scale=0.4] at (0,0) (0) {};
    \node[circle,scale=0.4] at (1*\s,-1) (1) {};
    \node[circle,scale=0.4] at (2*\s,-2) (2) {};
    \node[circle,scale=0.4] at (3*\s,-1) (3) {};
    \node[circle,scale=0.4] at (4*\s,-2) (4) {};
    \node[circle,scale=0.4] at (5*\s,-1) (5) {};
    \node[circle,scale=0.4] at (6*\s,-2) (6) {};
    \node[circle,scale=0.4] at (7*\s,-1) (7) {};
    \node[circle,scale=0.4] at (8*\s,-2) (8) {};
    \node[circle,scale=0.4] at (9*\s,-1) (9) {};
    \node[circle,scale=0.4] at (10*\s,0) (10) {};
    \draw (0)--(1)--(2)--(3)--(4)--(5)--(6)--(7)--(8)--(9)--(10);
    \node[] at (5*\s, -3) {\LARGE $LLRLRLRLRR$};
\end{tikzpicture}
}
\end{center}
\caption{Illustration of the upper prime $RRLRRLRRLLLL$, and the lower prime $LLRLRLRLRR$.}
\label{figureupperlower}
\end{figure}

\paragraph{Note:} An upper prime $P$ starts with the letter $R$ and ends with the letter $L$. Moreover, $P$ can be written as $RZL$ where $Z$ is a (possibly empty) product of upper primes. See Figure \ref{figureupperproduct}.

\begin{figure}[H]
\begin{center}
\scalebox{0.5}{
\begin{tikzpicture}
    \newcommand\s{0.4}
    \node[circle,scale=0.4] at (0,0) (0) {};
    \node[circle,scale=0.4] at (\s,1) (1) {};
    \node[circle,scale=0.4] at (2*\s,2) (2) {};
    \node[circle,scale=0.4] at (3*\s,1) (3) {};
    \node[circle,scale=0.4] at (4*\s,2) (4) {};
    \node[circle,scale=0.4] at (5*\s,3) (5) {};
    \node[circle,scale=0.4] at (6*\s,2) (6) {};
    \node[circle,scale=0.4] at (7*\s,1) (7) {};
    \node[circle,scale=0.4] at (8*\s,0) (8) {};
    \draw (0)--(1)--(2)--(3)--(4)--(5)--(6)--(7)--(8);
    \node[] at (4*\s, 0) {\LARGE $P$};
\end{tikzpicture}
\hspace{2cm}
\begin{tikzpicture}
    \newcommand\s{0.4}
    \node[circle,scale=0.4] at (0,0) (0) {};
    \node[circle,scale=0.4] at (\s,1) (1) {};
    \node[circle,scale=0.4] at (2*\s,2) (2) {};
    \node[circle,scale=0.4] at (3*\s,1) (3) {};
    \node[circle,scale=0.4] at (4*\s,2) (4) {};
    \node[circle,scale=0.4] at (5*\s,3) (5) {};
    \node[circle,scale=0.4] at (6*\s,2) (6) {};
    \node[circle,scale=0.4] at (7*\s,1) (7) {};
    \node[circle,scale=0.4] at (8*\s,0) (8) {};
    \draw (0)--(1)--(2)--(3)--(4)--(5)--(6)--(7)--(8);
    \draw[dashed] (-1,1)--(8*\s+1,1);
    \node[] at (4*\s, 0) {\LARGE $RZL$};
\end{tikzpicture}
\hspace{2cm}
\begin{tikzpicture}
    \newcommand\s{0.4}
    \node[circle,scale=0.4] at (0,1/2) (0) {};
    \node[circle,scale=0.4] at (\s,1) (1) {};
    \node[circle,scale=0.4] at (2*\s,2) (2) {};
    \node[circle,scale=0.4] at (3*\s,1) (3) {};
    \node[circle,scale=0.4] at (4*\s,2) (4) {};
    \node[circle,scale=0.4] at (5*\s,3) (5) {};
    \node[circle,scale=0.4] at (6*\s,2) (6) {};
    \node[circle,scale=0.4] at (7*\s,1) (7) {};
    \draw (1)--(2)--(3)--(4)--(5)--(6)--(7);
    \node[] at (4*\s, 0) {\LARGE $Z$};
\end{tikzpicture}
}
\end{center}
\caption{Visualization of the upper prime $P=RRLRRLLL$, and how it can be written as $RZL$, where $Z=(RL)(RRLL)$. The words $RL$ and $RRLL$ are upper primes.}
\label{figureupperproduct}
\end{figure}

\pagebreak
\section{A minimal generating set}

This section focuses on finding a minimal subset of $S$ that generates $\J$. By Lemma \ref{sprime}, the subset $S'\subset S$ generates $\J$. In this section, two more subsets, $S''$ and $S'''$, are defined, with $S'''\subset S''\subset S'\subset S$. The main result states that $S'''$ is a minimal subset of $S$ that generates $\J$. Because some choices are made when defining $S'''$, this result actually gives a whole family of minimal subsets of $S$ that generate $\J$. Showing that $S''$ generates $\J$ serves as an intermediate result towards the main goal.

\white
\begin{lemma}
\label{upperandlr}
Let $P$ and $Q$ be upper primes. Then $PQ-QP$ is in the ideal generated by
$$\{U(LR)-(LR)U\colon\; U \text{ is an upper prime}\}.$$
\end{lemma}

\begin{proof}
By Lemma \ref{idealswap}, it is enough to show that there is a sequence of swaps between $PQ$ and $QP$ where every swap is of type $(U,LR)$ for some upper prime $U$. In this proof, these kinds of swaps are called ``upper swaps''. The proof is an induction on $l(PQ)$. As $P$ and $Q$ are prime words, both are nonempty balanced words so both have length $\geq 2$. Therefore the smallest possible case is $l(PQ)=4$, with $P=Q=RL$. In this case $PQ=QP$ so the claim is true because no swaps are needed.

Next, assume that $l(PQ)\geq 6$. There exist nonnegative integers $p,q$ and upper primes $P_1,\dots, P_p$ and $Q_1,\cdots Q_q$ so that $P=RP_1\cdots P_p L$ and $Q=RQ_1\dots Q_q L$. The diagram below describes how the sequence of swaps between $PQ$ and $QP$ can be found. Each arrow in the diagram represents a sequence of swaps. The top arrow involves swaps of type $(Q_i,LR)$, and the bottom arrow swaps of type $(P_i,LR)$. Both of these are upper swaps.

The middle arrow involves swaps of type $(P_i,Q_j)$; notice that $l(P_iQ_j)\leq n-4$ for any $i\in \{1,\dots, p\}$ and $j\in \{1,\dots, q\}$. Therefore, by induction, there exists a sequence of swaps between $P_iQ_j$ and $Q_jP_i$ where every swap is an upper swap. This means that each swap of type $(P_i,Q_j)$ can be replaced by a sequence of upper swaps.

\begin{center}
\begin{tikzcd}
PQ=RP_1\dots P_p (LR)Q_1\dots Q_q L \arrow[d, leftrightarrow,"\text{\quad swaps of type } {(Q_i,LR)}"] \\
RP_1\dots P_pQ_1\dots Q_q (LR) L \arrow[d, leftrightarrow, "\text{\quad swaps of type } {(P_i,Q_j)} \text { (induction)}"] \\
RQ_1\dots Q_q P_1\dots P_p (LR) L \arrow[d, leftrightarrow,"\text{\quad swaps of type } {(P_i,LR)}"] \\
QP=RQ_1\dots Q_q (LR) P_1\dots P_p L
\end{tikzcd}
\end{center}

The three arrows together give the desired sequence of swaps between $PQ$ and $QP$, which completes the proof.
\end{proof}

\paragraph{Note:} The following figure gives a ``picture proof'' of Lemma \ref{upperandlr} in the case where $p=2$ and $q=1$.

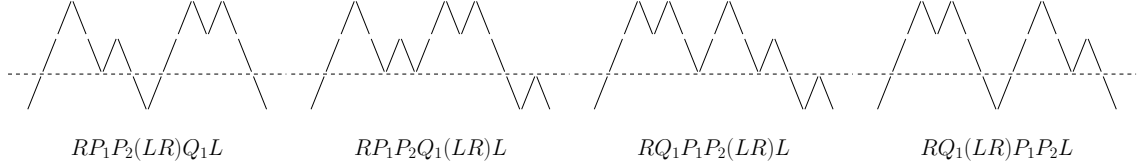
\begin{figure}[H]
\begin{center}
\scalebox{0.5}{
\begin{tikzpicture}
    \newcommand\s{0.4}
    %\node at (0,3) {\LARGE 1};
    \node[circle,scale=0.4] at (0,0) (0) {};
    \node[circle,scale=0.4] at (1*\s,1) (1) {};
    \node[circle,scale=0.4] at (2*\s,2) (2) {};
    \node[circle,scale=0.4] at (3*\s,3) (3) {};
    \node[circle,scale=0.4] at (4*\s,2) (4) {};
    \node[circle,scale=0.4] at (5*\s,1) (5) {};
    \node[circle,scale=0.4] at (6*\s,2) (6) {};
    \node[circle,scale=0.4] at (7*\s,1) (7) {};
    \node[circle,scale=0.4] at (8*\s,0) (8) {};
    \node[circle,scale=0.4] at (9*\s,1) (9) {};
    \node[circle,scale=0.4] at (10*\s,2) (10) {};
    \node[circle,scale=0.4] at (11*\s,3) (11) {};
    \node[circle,scale=0.4] at (12*\s,2) (12) {};
    \node[circle,scale=0.4] at (13*\s,3) (13) {};
    \node[circle,scale=0.4] at (14*\s,2) (14) {};
    \node[circle,scale=0.4] at (15*\s,1) (15) {};
    \node[circle,scale=0.4] at (16*\s,0) (16) {};
    \draw (0)--(1)--(2)--(3)--(4)--(5)--(6)--(7)--(8)--(9)--(10)--(11)--(12)--(13)--(14)--(15)--(16);
    \draw[dashed] (-1/2,1)--(16*\s+1/2,1);
    \node[] at (8*\s, -1) {\LARGE $RP_1P_2(LR)Q_1L$};
    %\node[] at (1/2*\s, -1) {\LARGE $R$};
    %\node[] at (3*\s, -1) {\LARGE $P_1$};
    %\node[] at (5.5*\s, -1) {\LARGE $P_2$};
    %\node[] at (8.5*\s, -1) {\LARGE $(LR)$};
    %\node[] at (12*\s, -1) {\LARGE $Q_1$};
    %\node[] at (15.5*\s, -1) {\LARGE $L$};
\end{tikzpicture}
\begin{tikzpicture}
    \newcommand\s{0.4}
    %\node at (0,3) {\LARGE 2};
    \node[circle,scale=0.4] at (0,0) (0) {};
    \node[circle,scale=0.4] at (1*\s,1) (1) {};
    \node[circle,scale=0.4] at (2*\s,2) (2) {};
    \node[circle,scale=0.4] at (3*\s,3) (3) {};
    \node[circle,scale=0.4] at (4*\s,2) (4) {};
    \node[circle,scale=0.4] at (5*\s,1) (5) {};
    \node[circle,scale=0.4] at (6*\s,2) (6) {};
    \node[circle,scale=0.4] at (7*\s,1) (7) {};
    \node[circle,scale=0.4] at (8*\s,2) (8) {};
    \node[circle,scale=0.4] at (9*\s,3) (9) {};
    \node[circle,scale=0.4] at (10*\s,2) (10) {};
    \node[circle,scale=0.4] at (11*\s,3) (11) {};
    \node[circle,scale=0.4] at (12*\s,2) (12) {};
    \node[circle,scale=0.4] at (13*\s,1) (13) {};
    \node[circle,scale=0.4] at (14*\s,0) (14) {};
    \node[circle,scale=0.4] at (15*\s,1) (15) {};
    \node[circle,scale=0.4] at (16*\s,0) (16) {};
    \draw (0)--(1)--(2)--(3)--(4)--(5)--(6)--(7)--(8)--(9)--(10)--(11)--(12)--(13)--(14)--(15)--(16);
    \draw[dashed] (-1/2,1)--(16*\s+1/2,1);
    \node[] at (8*\s, -1) {\LARGE $RP_1P_2Q_1(LR)L$};
    % \node[] at (1/2*\s, -1) {\LARGE $R$};
    % \node[] at (3*\s, -1) {\LARGE $P_1$};
    % \node[] at (6*\s, -1) {\LARGE $P_2$};
    % \node[] at (10*\s, -1) {\LARGE $Q_1$};
    % \node[] at (13.5*\s, -1) {\LARGE $(LR)$};
    % \node[] at (15.5*\s, -1) {\LARGE $L$};
\end{tikzpicture}
\begin{tikzpicture}
    \newcommand\s{0.4}
    %\node at (0,3) {\LARGE 3};
    \node[circle,scale=0.4] at (0,0) (0) {};
    \node[circle,scale=0.4] at (1*\s,1) (1) {};
    \node[circle,scale=0.4] at (2*\s,2) (2) {};
    \node[circle,scale=0.4] at (3*\s,3) (3) {};
    \node[circle,scale=0.4] at (4*\s,2) (4) {};
    \node[circle,scale=0.4] at (5*\s,3) (5) {};
    \node[circle,scale=0.4] at (6*\s,2) (6) {};
    \node[circle,scale=0.4] at (7*\s,1) (7) {};
    \node[circle,scale=0.4] at (8*\s,2) (8) {};
    \node[circle,scale=0.4] at (9*\s,3) (9) {};
    \node[circle,scale=0.4] at (10*\s,2) (10) {};
    \node[circle,scale=0.4] at (11*\s,1) (11) {};
    \node[circle,scale=0.4] at (12*\s,2) (12) {};
    \node[circle,scale=0.4] at (13*\s,1) (13) {};
    \node[circle,scale=0.4] at (14*\s,0) (14) {};
    \node[circle,scale=0.4] at (15*\s,1) (15) {};
    \node[circle,scale=0.4] at (16*\s,0) (16) {};
    \draw (0)--(1)--(2)--(3)--(4)--(5)--(6)--(7)--(8)--(9)--(10)--(11)--(12)--(13)--(14)--(15)--(16);
    \draw[dashed] (-1/2,1)--(16*\s+1/2,1);
    \node[] at (8*\s, -1) {\LARGE $RQ_1P_1P_2(LR)L$};
    % \node[] at (1/2*\s, -1) {\LARGE $R$};
    % \node[] at (4*\s, -1) {\LARGE $Q_1$};
    % \node[] at (9*\s, -1) {\LARGE $P_1$};
    % \node[] at (11.3*\s, -1) {\LARGE $P_2$};
    % \node[] at (13.8*\s, -1) {\LARGE $(LR)$};
    % \node[] at (15.8*\s, -1) {\LARGE $L$};
\end{tikzpicture}
\begin{tikzpicture}
    \newcommand\s{0.4}
    %\node at (0,3) {\LARGE 4};
    \node[circle,scale=0.4] at (0,0) (0) {};
    \node[circle,scale=0.4] at (1*\s,1) (1) {};
    \node[circle,scale=0.4] at (2*\s,2) (2) {};
    \node[circle,scale=0.4] at (3*\s,3) (3) {};
    \node[circle,scale=0.4] at (4*\s,2) (4) {};
    \node[circle,scale=0.4] at (5*\s,3) (5) {};
    \node[circle,scale=0.4] at (6*\s,2) (6) {};
    \node[circle,scale=0.4] at (7*\s,1) (7) {};
    \node[circle,scale=0.4] at (8*\s,0) (8) {};
    \node[circle,scale=0.4] at (9*\s,1) (9) {};
    \node[circle,scale=0.4] at (10*\s,2) (10) {};
    \node[circle,scale=0.4] at (11*\s,3) (11) {};
    \node[circle,scale=0.4] at (12*\s,2) (12) {};
    \node[circle,scale=0.4] at (13*\s,1) (13) {};
    \node[circle,scale=0.4] at (14*\s,2) (14) {};
    \node[circle,scale=0.4] at (15*\s,1) (15) {};
    \node[circle,scale=0.4] at (16*\s,0) (16) {};
    \draw (0)--(1)--(2)--(3)--(4)--(5)--(6)--(7)--(8)--(9)--(10)--(11)--(12)--(13)--(14)--(15)--(16);
    \draw[dashed] (-1/2,1)--(16*\s+1/2,1);
    \node[] at (8*\s, -1) {\LARGE $RQ_1(LR)P_1P_2L$};
    % \node[] at (1/2*\s, -1) {\LARGE $R$};
    % \node[] at (4*\s, -1) {\LARGE $Q_1$};
    % \node[] at (8*\s, -1) {\LARGE $(LR)$};
    % \node[] at (11*\s, -1) {\LARGE $P_1$};
    % \node[] at (14*\s, -1) {\LARGE $P_2$};
    % \node[] at (15.5*\s, -1) {\LARGE $L$};
\end{tikzpicture}
}
\end{center}
    \caption{Illustration of the swaps in the proof of Lemma \ref{upperandlr}, for $p=2$ and $q=1$. First $LR$ is moved past $Q_1$, then the $Q_1$ is moved past both $P_2$ and $P_1$ (this is where induction is used), and finally $LR$ is moved back to the center.}
    \label{figurepictureproof}
\end{figure}

\white
\begin{lemma}
\label{lowerandrl}
Let $P$ and $Q$ be lower primes. Then $PQ-QP$ is in the ideal generated by
$$\{(RL)D-D(RL)\colon\; D \text{ is a lower prime}\}.$$
\end{lemma}
\begin{proof}
Similar to Lemma \ref{upperandlr}.
\end{proof}

\paragraph {Yet another generating set:} Define
$$S''=\{UD-DU\colon \; U \text{ is an upper prime, } D \text{ is a lower prime}\}.$$
Note that $S''\subset S'\subset S$.

\white
\begin{proposition}
\label{sdoubleprime}
The set $S''$ generates $\J$.
\end{proposition}

\begin{proof}
Let $\J''$ denote the ideal generated by $S''$. The goal is to show that $\J=\J''$. The inclusion $\J''\subset\J$ follows from the fact that $S''\subset S$. By Lemma \ref{sprime}, the set $S'$ generates $\J$, so for the reverse inclusion, it is enough to show that $S'\subset \J''$. This translates to the following claim: if $P$ and $Q$ are prime words, then $PQ-QP\in\J''$. There are four cases, depending on whether $P$ and $Q$ are upper or lower primes. The table below shows why $PQ-QP\in \J''$ in each of the cases.

\begin{center}
\begin{tabular}{|m{2.4cm}|m{5.9cm}|m{5.9cm}|}
    \hline
    & $Q$ upper prime & $Q$ lower prime  \\
    \hline
    $P$ upper prime & By Lemma \ref{upperandlr}, $PQ-QP$ is in the ideal generated by elements of the type $U(LR)-(LR)U$ where $U$ is an upper prime, and these are all elements of $S''$.& $PQ-QP\in S''\subset \J''$.\\
    \hline
    $P$ lower prime & Combining $PQ-QP=-(QP-PQ)$ and the fact that $QP-PQ\in S''$ gives $PQ-QP\in \J''$. & By Lemma \ref{lowerandrl}, $PQ-QP$ is in the ideal generated by elements of the type $(RL)D-D(RL)$ where $D$ is a lower prime, and these are all elements of $S''$.\\
    \hline
\end{tabular}    
\end{center}
Note that an ideal generated by elements of $S''$ is a subset of $\J''$. It has now been shown that $PQ-QP\in \J''$, which concludes the proof.
\end{proof}

\white
\begin{samepage}
\begin{corollary}
\label{upperlowerswap}
Let $X$ and $Y$ be any words. The following are equivalent:
\begin{enumerate}[label=(\roman*)]
    \item $X\sim Y$.
    \item There is a sequence of swaps between $X$ and $Y$, where every swap is of type $(U,D)$ for some upper prime $U$ and a lower prime $D$.
\end{enumerate}
\end{corollary}
\end{samepage}
\begin{proof}
Use Lemma \ref{idealswap} and Proposition \ref{equivalence}, together with the fact that $S''$ generates $\J$ from Proposition \ref{sdoubleprime}.
\end{proof}

\white
\begin{lemma}
\label{insidelemma}
Let $U$ be an upper prime, $D$ a lower prime, and $W$ a balanced word. If $UD$ and $W$ are related by a swap, then exactly one of the following holds:
\begin{enumerate}[label=(\roman*)]
    \item $W=DU$;
    \item $W=U'D$ where $U$ and $U'$ are related by a swap;
    \item $W=UD'$ where $D$ and $D'$ are related by a swap. 
\end{enumerate}
\end{lemma}

\begin{proof}
Because $UD$ and $W$ are related by a swap, there are nonempty balanced words $F,G$ and words $W_1,W_2$ so that $UD=W_1FGW_2$ and $W=W_1GFW_2$. Note that the product $FG$ starts and ends at the same elevation of $UD$. (In other words, if $a=l(W_1)$ and $b=l(FG)$, then $e_a(UD)=e_{a+b}(UD)$.) The table below shows how the elevation $e_k(UD)$ behaves for $0\leq k\leq l(UD)$.

\begin{center}
\begin{tabular}{|*{6}{c|}}
    \hline
    $k$ & 0 & $1,\dots, l(U)-1$ & $l(U)$ & $l(U)+1,\dots, l(U)+l(D)-1$ & $l(U)+l(D)=l(UD)$ \\
    \hline
    $e_k(UD)$ & 0 & $>0$ & 0 & $<0$ & 0 \\
    \hline
\end{tabular}
\end{center}

If $FG$ starts at elevation zero, then both $F$ and $G$ start and end at elevation zero. In this case, the only possibility is $U=F$ and $D=G$, because elevation zero appears exactly $3$ times in $UD$. This is case (i). If $FG$ starts at positive elevation, then $FG$ is a subword of $U$, which gives case (ii). If $FG$ starts at negative elevation, then $FG$ is a subword of $D$, which gives case (iii).
\end{proof}

\paragraph{Note:} The picture below illustrates the table in the proof of Lemma \ref{insidelemma}.

\begin{figure}[H]
\begin{center}
\scalebox{0.5}{
\begin{tikzpicture}
    \newcommand\s{0.4}
    \node[circle,fill = black] at (0,0) (0) {};
    \node[circle,scale=0.4] at (1*\s,1) (1) {};
    \node[circle,scale=0.4] at (2*\s,2) (2) {};
    \node[circle,scale=0.4] at (3*\s,3) (3) {};
    \node[circle,scale=0.4] at (4*\s,2) (4) {};
    \node[circle,scale=0.4] at (5*\s,1) (5) {};
    \node[circle,scale=0.4] at (6*\s,2) (6) {};
    \node[circle,scale=0.4] at (7*\s,1) (7) {};
    \node[circle,fill = black] at (8*\s,0) (8) {};
    \node[circle,scale=0.4] at (9*\s,-1) (9) {};
    \node[circle,scale=0.4] at (10*\s,-2) (10) {};
    \node[circle,scale=0.4] at (11*\s,-1) (11) {};
    \node[circle,scale=0.4] at (12*\s,-2) (12) {};
    \node[circle,scale=0.4] at (13*\s,-1) (13) {};
    \node[circle,scale=0.4] at (14*\s,-2) (14) {};
    \node[circle,scale=0.4] at (15*\s,-1) (15) {};
    \node[circle,fill = black] at (16*\s,0) (16) {};
    \draw (0)--(1)--(2)--(3)--(4)--(5)--(6)--(7)--(8)--(9)--(10)--(11)--(12)--(13)--(14)--(15)--(16);
    \draw[dashed] (-1,0)--(16*\s+1,0);
    %\node[] at (3*\s, -1) {\LARGE $RRLRLL$};
\end{tikzpicture}
}
\end{center}
    \caption{Illustration of $UD=(RRRLLRLL)(LLRLRLRR)$. The dashed line represents elevation zero. Elevation is positive inside $U$, negative inside $D$, and zero for exactly three indices; the zero elevation locations are marked with a black dot.}
    \label{figureud}
\end{figure}
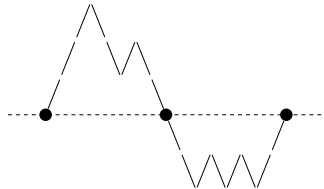

\white
\begin{lemma}
\label{prime-equiv}
The equivalence class of any upper prime consists entirely of upper primes of the same length, and the equivalence class of any lower prime consists entirely of lower primes of the same length.
\end{lemma}
\begin{proof} Let $U$ be an upper prime, let $n=l(U)$, and let $W$ be a word such that $U\sim W$. The goal is to show that $W$ is an upper prime of length $n$. By Proposition \ref{equivalence}, there is a sequence of swaps between $U$ and $W$, so $W$ is a balanced word of length $n$ and thus $e_0(W)=e_n(W)=0$. Zero appears in $E(U)$ exactly twice, and $E(U)=E(W)$ by Lemma \ref{equivmultiset}. Therefore $e_k(W)>0$ for all $k$ with $1\leq k\leq n-1$, which means that $W$ is also an upper prime. This proves the statement about upper primes, and the statement about lower primes can be proven similarly.
\end{proof}

\paragraph{Note:} The following proposition gives an equivalent condition for a subset of $S$ to generate $\J$. This will be a key ingredient in the proof of the main result.

\white
\begin{proposition}
\label{either}
Let $S^\star\subset S$. Then $S^\star$ generates $\J$ if and only if for any upper prime $U$ and lower prime $D$, there exist words $U'\sim U$ and $D'\sim D$ such that either $U'D'-D'U'\in S^\star$ or $D'U'-U'D'\in S^\star$.
\end{proposition}
\begin{proof}
As $S^\star\subset S$, there is a set $I$ and balanced words $F_i,G_i$ so that $S^\star=\{F_iG_i-G_iF_i\}_{i\in I}$. For the purpose of this proof, a swap of type $(F_i,G_i)$ for some $i$ is called an $S^\star$-swap.

Assume first that $S^\star$ generates $\J$. If the second condition in the statement does not hold, then there exist an upper prime $U$ and a lower prime $D$ so that for any $U'\sim U$ and $D'\sim D$, the elements $U'D'-D'U'$ and $D'U'-U'D'$ are not in $S^\star$. In other words, swaps of type $(U',D')$ are not $S^\star$-swaps. Because $UD-DU\in \J$, there is a sequence of $S^\star$-swaps between $UD$ and $DU$, by Lemma \ref{idealswap}.

Let $\{U_j\}_{j=1}^m$ and $\{D_k\}_{k=1}^n$ be the equivalence classes of $U$ and $D$, respectively. Let $X$ be a word so that there is a sequence of swaps between $UD$ and $X$. By repeated application of Lemma \ref{insidelemma}, $X$ can be $U_jD_k$ for some $j$ and $k$, but because swaps of type $(U_j,D_k)$ are not $S^\star$-swaps, $X$ cannot be $D_kU_j$. In particular, $X$ cannot be $DU$, which is a contradiction.

Now assume that for any upper prime $U$ and lower prime $D$, there exist words $U'\sim U$ and $D'\sim D$ such that either $U'D'-D'U'\in S^\star$ or $D'U'-U'D'\in S^\star$, or in other words, the swap of the type $(U',D')$ is an $S^\star$-swap. By Proposition \ref{sdoubleprime}, $S''$ generates $\J$, so for showing that $S^\star$ generates $\J$ it is enough to show that every element of $S''$ is in the ideal generated by $S^\star$. Using Lemma \ref{idealswap}, this translates to showing that for any upper prime $U$ and lower prime $D$, there is a sequence of $S^\star$-swaps between $UD$ and $DU$.

The proof is by induction on $l(UD)$. The smallest possible case is $l(UD)=4$; this happens only for $U=RL$ and $D=LR$. Both of these words are the only elements of their equivalence classes, so the assumption implies that the swap of type $(U,D)$ is an $S^\star$-swap and by itself forms the desired sequence of $S^\star$-swaps between $UD$ and $DU$.

Assume now that $l(UD)>4$. By the assumption, there exist words $U'\sim U$ and $D'\sim D$ such that the swap of type $(U',D')$ is an $S^\star$-swap. By Corollary \ref{upperlowerswap}, there is a sequence of swaps between $U$ and $U'$ where each swap is of type $(U'', D'')$ for some upper prime $U''$ and lower prime $D''$. Now $l(U''D'')\leq l(U)\leq l(UD)-2$, so by induction, there is a sequence of $S^\star$-swaps between $U''D''$ and $D''U''$. Combining the above observations gives a sequence of $S^\star$-swaps between $U$ and $U'$. Similarly, there is a sequence of $S^\star$-swaps between $D$ and $D'$.

Now there is a sequence of $S^\star$-swaps between $UD$ and $DU$ as follows, with each arrow representing a sequence of $S^\star$-swaps: $UD\leftrightarrow U'D' \leftrightarrow D'U' \leftrightarrow DU$. This concludes the proof.
\end{proof}

\paragraph{Representatives:} Let $\Upsilon$ denote the set of equivalence classes of upper primes, and $\Lambda$ the set of equivalence classes of lower primes. In order to state the main result, a representative will be chosen for each equivalence class. For each $\upsilon \in \Upsilon$, denote the chosen representative by $U_\upsilon$, and for each $\lambda\in \Lambda$, denote the chosen representative by $D_\lambda$.

\paragraph{Minimal generating set:} Define the set
$$S'''=\big\{U_\upsilon D_\lambda-D_\lambda U_\upsilon\colon\, \upsilon \in \Upsilon, \lambda\in \Lambda\big\}.$$
Note that $S'''$ depends on the chosen representatives.

\white
\begin{theorem}
\label{minimal}
$S'''$ is a minimal subset of $S$ that generates $\J$.
\end{theorem}

\begin{proof}
Firstly, clearly $S'''\subset S$. The way that $S'''$ is defined guarantees that the condition in Proposition \ref{either} is satisfied; by this proposition, $S'''$ generates $\J$. By the same proposition, no relation $U_\upsilon D_\lambda-D_\lambda U_\upsilon$ can be removed from $S'''$, as then $S'''$ wouldn't generate $\J$ anymore. This concludes the proof.
\end{proof}

\pagebreak

\section{Choosing representatives}

As was shown in Theorem \ref{minimal}, the set $S'''$ is a minimal set of generators for $\J$. Note that $S'''$ is not unique, because it depends on the choice of representatives for the equivalence classes of prime words; therefore the result actually gives a family of minimal sets of generators. This section focuses on one possible choice of representatives: the minimal word with respect to alphabetical order. The concept of a reduced word is introduced, and it will be shown that every equivalence class of balanced words contains a unique reduced word which coincides with the minimal word. The reduced word can be easily found using an algorithm, and this gives a convenient way of finding the minimal word of the equivalence class of a given balanced word, without having to list all the words in the equivalence class.

\paragraph{Definition (alphabetical order):} The set of words is linearly ordered by alphabetical order. Let $X$ and $Y$ be any words. The notation $X<Y$ means that $X$ comes before $Y$ in alphabetical order; for example, $LLRR<RRLL$. Some other examples are $LL<LLR$ and $LLL<LR$.

\paragraph{Minimal representatives:} By Corollary \ref{finite}, each equivalence class of words is finite. Therefore, each equivalence class has a minimal word with respect to alphabetical order. Moreover, because alphabetical order is a linear order, this minimal word is unique and it will thus be referred to as \textit{the minimal word} of the equivalence class. In this section, the representatives $U_\upsilon$ and $D_\lambda$ needed to define $S'''$ are chosen to be the minimal words of their equivalence classes. (See definition of $S'''$ right before Theorem \ref{minimal}.) 

\paragraph{Definition (reduced words):} A word $W$ is called \textit{reduced}, if it does not contain a subword of the type $UD$ where $U$ is an upper prime and $D$ is a lower prime.

\white
\begin{lemma}
\label{reductionlemma}
Let $U$ be an upper prime and $D$ a lower prime. Let $W_1$ and $W_2$ be any words. Then $W_1DUW_2\sim W_1UDW_2$ and $W_1DUW_2<W_1UDW_2$.
\end{lemma}
\begin{samepage}
\begin{proof}
The words are related by a swap of type $(U,D)$, which implies equivalence by Proposition \ref{equivalence}. The word $U$ starts with the letter $R$ and $D$ with the letter $L$, so $DU<UD$ and therefore also $W_1DUW_2<W_1UDW_2$.
\end{proof}
\end{samepage}

\white
\begin{samepage}
\begin{lemma}
\label{minimal-red}
The minimal word of an equivalence class is reduced.
\end{lemma}
\begin{proof}
If the minimal word is not reduced, then it has a subword $UD$ where $U$ is an upper prime and $D$ is a lower prime. Then a swap of type $(U,D)$ produces a smaller word in the same equivalence class, by Lemma \ref{reductionlemma}. This is a contradiction, so the minimal word is reduced.
\end{proof}
\end{samepage}

\paragraph{Unique reduced word:} The next results are preparation for Proposition \ref{unique-red}, which states that every equivalence class of balanced words contains a unique reduced word.

\pagebreak

%\white
\begin{lemma}
\label{technical}
Let $W$ be a balanced word, and assume that $e_i(W)\geq e_j(W)$ for some $i<j$. Then there exists an integer $e\in [e_j(W), e_i(W)]$, and integers $i'\in [0,i]$ and $j'\in [j,l(W)]$ with $e_{i'}(W)=e_{j'}(W)=e$. (The bracket notation refers to a closed interval.)
\end{lemma}
\begin{proof}
There are three cases:
\begin{enumerate}[label =(\roman*)]
    \item $e_i(W)\geq e_j(W)\geq 0$: In this case, choose $e=e_j(W)$ and $j'=j$. As $e_0(W)=0\leq e\leq e_i(W)$, there exists $i'\in [0,i]$ so that $e_{i'}(W)=e=e_{j'}(W)$.
    \item $e_i(W)>0>e_j(W)$: In this case, choose $e=0$, $i'=0$, and $j'=l(W)$.
    \item $0\geq e_i(W)\geq e_j(W)$: Similar to (i), with $e=e_i(W)$ and $i'=i$.
\end{enumerate}
\end{proof}

\paragraph{Example:} As an example to Lemma \ref{technical}, consider the balanced word $W=RRRLLRLL$, and let $i=3$ and $j=5$. For this word, $e_3(W)=3\geq 1=e_5(W)$, so the lemma implies that some elevation between 1 and 3 is obtained at both ends. For this example, this elevation is either $e=2$ (with $i'=2$ and $j'=6$), or $e=1$ (with $i'=1$ and $j'\in\{5,7\}$). See illustration in Figure \ref{figuretechnical}.

\begin{figure}[H]
\begin{center}
\scalebox{0.5}{
\begin{tikzpicture}
    \newcommand\s{0.4}
    \node[circle,scale=0.4] at (0,0) (0) {};
    \node[circle,scale=0.4] at (\s,1) (1) {};
    \node[circle,scale=0.4] at (2*\s,2) (2) {};
    \node[circle, fill = black] at (3*\s,3) (3) {};
    \node[circle,scale=0.4] at (4*\s,2) (4) {};
    \node[circle, fill = black] at (5*\s,1) (5) {};
    \node[circle,scale=0.4] at (6*\s,2) (6) {};
    \node[circle,scale=0.4] at (7*\s,1) (7) {};
    \node[circle,scale=0.4] at (8*\s,0) (8) {};
    \draw (0)--(1)--(2)--(3)--(4)--(5)--(6)--(7)--(8);
    \draw[dashed] (-1,1)--(8*\s+1,1);
    \draw[dashed] (-1,3)--(8*\s+1,3);
    \node[] at (-3,1) {\LARGE $e_5(W)=1$};
    \node[] at (-3,3) {\LARGE $e_3(W)=3$};
\end{tikzpicture}
}
\end{center}
\caption{An example of Lemma \ref{technical}. The black circles indicate the locations $i=3$ and $j=5$, and the horizontal lines the elevations $e_3(W)=3$ and $e_5(W)=1$. Graphically, the statement of the lemma says that it is possible to draw a horizontal line between the two given horizontal lines which intersects the graph both left from $i$ and right from $j$; here the line can be drawn at elevation 1 or 2.}
\label{figuretechnical}
\end{figure}
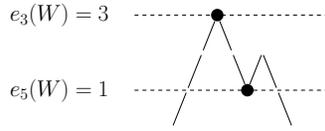

\white
\begin{proposition}
\label{reduced-equiv}
Let $W$ be a balanced word. The following are equivalent:
\begin{enumerate}[label=(\roman*)]
    \item $W$ is reduced;
    \item $W$ does not contain a subword of the type $RL^nR$ where $n\geq 2$;
    \item $W=L^a(RL)^{k_1}R(RL)^{k_2}R\cdots (RL)^{k_m}RL^b$, where $a,b,m \geq 0$, with $a+b=m$, and $k_i\geq 0$ for $1\leq i\leq m$.
\end{enumerate}
\end{proposition}
\begin{proof}
(i) $\implies$ (ii): Assume, on the contrary, that $W$ contains a subword $RL^nR$ for some $n\geq 2$; this means that $W=W_1RL^nRW_2$ for some words $W_1,W_2$. Let $i=l(W_1)$, and $j=i+n+2$ (indices $i$ and $j$ are the starting and ending locations of the subword $RL^nR$). Now $e_i(W)=e_j(W)+n-2\geq e_j(W)$, because $n\geq 2$. By Lemma \ref{technical}, there exists an integer $e$ with $e_j(W)\leq e\leq e_i(W)$, together with $i'\leq i$ and $j'\geq j$ so that $e_{i'}(W)=e$ and $e_{j'}(W)=e$. If there are multiple choices for $i'$ and $j'$, then the largest possible $i'$ and smallest possible $j'$ are chosen. But now the subword of $W$ starting at index $i'$ and ending at $j'$ is a product $UD$ where $U$ is an upper prime and $D$ is a lower prime, which contradicts (i). (See example picture in Figure \ref{figuresubword}.)

(ii) $\implies$ (iii): Write $W=L^aW'L^b$ where $a,b\geq 0$ and there are no letters $L$ in the beginning or end of $W$. If $W'$ is empty, then $W$ is empty as well, because if there are no letters $R$ then there are no letters $L$ either, as $W$ is balanced. In this case $W$ is of the form (iii) with $a=b=m=0$. Assume $W'$ is not empty; then $W'$ starts and ends with a letter $R$. (The special case $W'=R$ can be dealt with in the same way as the general case with the two $R$'s different.)

In the word $W'$, any $L$ has an $R$ left to it, because two or more successive letters $L$ would produce a subword of type $RL^nR$ with $n\geq 2$. This means that $W'$ is a product of factors each of which is either $RL$ or $R$. By inserting the empty subword in the form $(RL)^0$ between successive letters $R$ (and on the left side of the leftmost $R$ if needed), it is possible to write $W'=(RL)^{k_1}R \cdots (RL)^{k_m}R$, for some $m\geq 1$, with $k_i\geq 0$ for $0\leq i\leq m$. It has been shown that $W=L^a(RL)^{k_1}R(RL)^{k_2}R\cdots (RL)^{k_m}RL^b$, with $a,b,m \geq 0$, and $k_i\geq 0$ for $1\leq i\leq m$. Finally, each factor $(RL)^{k_i}R$ has one more $R$ than $L$, and therefore $a+b=m$, because $W$ is balanced.

(iii) $\implies$ (i): If $W$ is not reduced, it contains a subword $UD$ where $U$ is an upper prime and $D$ is a lower prime. By definition, the words $U$ and $D$ both have length at least 2, so $U$ ends with an $L$ and $D$ starts with an $L$, and both $U$ and $L$ contain at least one $R$. This is a contradiction, as words of form (iii) do not contain adjacent letters $L$ in the middle.
\end{proof}

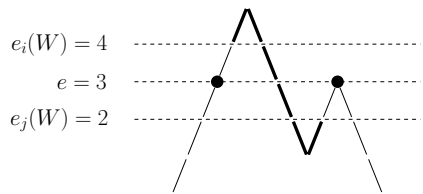
\begin{figure}[H]
\begin{center}
\scalebox{0.5}{
\begin{tikzpicture}
    \newcommand\s{0.4}
    \node[circle,scale=0.4] at (0,0) (0) {};
    \node[circle,scale=0.4] at (1*\s,1) (1) {};
    \node[circle,scale=0.4] at (2*\s,2) (2) {};
    \node[circle,fill = black] at (3*\s,3) (3) {};
    \node[circle,scale=0.4] at (4*\s,4) (4) {};
    \node[circle,scale=0.4] at (5*\s,5) (5) {};
    \node[circle,scale=0.4] at (6*\s,4) (6) {};
    \node[circle,scale=0.4] at (7*\s,3) (7) {};
    \node[circle,scale=0.4] at (8*\s,2) (8) {};
    \node[circle,scale=0.4] at (9*\s,1) (9) {};
    \node[circle,scale=0.4] at (10*\s,2) (10) {};
    \node[circle,fill = black] at (11*\s,3) (11) {};
    \node[circle,scale=0.4] at (12*\s,2) (12) {};
    \node[circle,scale=0.4] at (13*\s,1) (13) {};
    \node[circle,scale=0.4] at (14*\s,0) (14) {};
    \draw (0)--(1)--(2)--(3)--(4)--(5)--(6)--(7)--(8)--(9)--(10)--(11)--(12)--(13)--(14);
    \draw[line width = 2.5] (4)--(5)--(6)--(7)--(8)--(9)--(10); 
    \draw[dashed] (-1,3)--(14*\s+1,3);
    \draw[dashed] (-1,4)--(14*\s+1,4);
    \draw[dashed] (-1,2)--(14*\s+1,2);
    \node[] at (-3,2) {\LARGE $e_j(W)=2$};
    \node[] at (-2.4,3) {\LARGE $e=3$};
    \node[] at (-3,4) {\LARGE $e_i(W)=4$};
\end{tikzpicture}
}
\end{center}
\caption{Illustration of Proposition \ref{reduced-equiv} (i) $\implies$ (ii), with the example word $RRRRRLLLLRRLLL$ which has a subword $RL^4R$. In this example $i=4$ and $j=10$, and the subword $RL^4R$ is drawn in bold. The black circles indicate the locations $i'=3$ and $j'=11$, which are the starting and ending points for $UD=(RRLL)(LLRR)$. (Note that the choices $e=2$, $i'=2$, and $j'=10$ would be valid as well.)}
\label{figuresubword}
\end{figure}

\paragraph{Example:} The table below contains four reduced words, together with the corresponding parameters $a$, $b$, $m$, and $k_1,\dots, k_m$ from of Proposition \ref{reduced-equiv} (iii).

\begin{center}    
\begin{tabular}{|*{5}{c|}}
    \hline
    Reduced word & $a$ & $b$ & $m$ & $(k_1,\dots, k_m)$ \\
    \hline
    $LRRL$ & 1 & 1 & 2 & $(0,0)$ \\
    \hline
    $RLRRLL$ & 0 & 2 & 2 & $(1,0)$ \\
    \hline
    $LLRLRRRLRRLL$ & 2 & 2 & 4 & $(1,0,1,0)$ \\
    \hline
    $LRRRLRLRLL$ & 1 & 2 & 3 & $(0,0,2)$ \\
    \hline
\end{tabular}
\end{center}

\pagebreak

%\white
\begin{lemma}
\label{multisetvalues}
Let $W$ be a balanced word as in Proposition \ref{reduced-equiv} (iii). Then the elevation multiset of $W$ \nolinebreak is $$E(W)=\{0,-1,\dots, -a\}\cup \{0,1,\dots, b\}\cup \{(i-a)^{\mu_i}\}_{i=0}^m,$$
where the multiplicity $\mu_i$ of the value $i-a$ is given by
$$\mu_i=\begin{cases}k_1, \quad & i = 0\\
k_i+1+k_{i+1}, &1\leq i\leq m-1\\
k_m,&i=m.\end{cases}$$
\begin{proof}
The parts $L_a$ and $L_b$ of $W$ contribute to the multisets $\{0,-1,\dots,-a\}$ and $\{0,1,\dots, b\}$, respectively. It remains to show that the middle part $(RL)^{k_1}R(RL)^{k_2}R\cdots(RL)^{k_m}R$, excluding its endpoints, contributes to a multiset as in the statement.

First, the lowest elevation $-a\,(=0-a)$ appears after the first occurrence (which was included in $\{0,-1,\dots, -a\}$) precisely $k_1$ times. Similarly, the highest elevation $b\,(=m-a)$ appears before the last occurrence (included in $\{0,1,\dots, b\}$) exactly $k_m$ times. For $1\leq i\leq m-1$, the elevation $i-a$ appears $k_i$ times ``at the peaks'' of $(RL)^{k_i}$, then once more, and finally $k_{i+1}$ times ``after the peaks'' of $(RL)^{k_{i+1}}$, giving multiplicity $k_i+1+k_{i+1}$. See Figure \ref{figuremultisetexample} for an example.
\end{proof}

\begin{figure}[H]
\begin{center}
\scalebox{0.5}{
\begin{tikzpicture}
    \newcommand\s{0.4}
    \node[circle,scale=0.4] at (0,0) (0) {};
    \node[circle,scale=0.4] at (1*\s,-1) (1) {};
    \node[circle,scale=0.4] at (2*\s,-2) (2) {};
    \node[circle,scale=0.8,fill = black] at (3*\s,-1) (3) {};
    \node[circle,scale=0.4] at (4*\s,-2) (4) {};
    \node[circle,scale=0.8,fill = black] at (5*\s,-1) (5) {};
    \node[circle,scale=0.4] at (6*\s,-2) (6) {};
    \node[circle,scale=0.8,fill = black] at (7*\s,-1) (7) {};
    \node[circle,scale=0.4] at (8*\s,-2) (8) {};
    \node[circle,scale=0.4] at (9*\s,-1) (9) {};
    \node[circle,scale=0.8,fill = black] at (10*\s,0) (10) {};
    \node[circle,scale=0.4] at (11*\s,-1) (11) {};
    \node[circle,scale=0.8,fill = black] at (12*\s,0) (12) {};
    \node[circle,scale=0.4] at (13*\s,-1) (13) {};
    \node[circle,scale=0.4] at (14*\s,0) (14) {};
    \node[circle,scale=0.8,fill = black] at (15*\s,1) (15) {};
    \node[circle,scale=0.4] at (16*\s,0) (16) {};
    \node[circle,scale=0.4] at (17*\s,1) (17) {};
    \node[circle,scale=0.4] at (18*\s,2) (18) {};
    \node[circle,scale=0.4] at (19*\s,1) (19) {};
    \node[circle,scale=0.4] at (20*\s,0) (20) {};
    \draw (0)--(1)--(2)--(3)--(4)--(5)--(6)--(7)--(8)--(9)--(10)--(11)--(12)--(13)--(14)--(15)--(16)--(17)--(18)--(19)--(20);
    \draw [semithick] (11*\s,-1) circle (3pt);
    \draw [semithick] (13*\s,-1) circle (3pt);
    \draw [semithick] (16*\s,0) circle (3pt);
    \draw[dashed] (2.5*\s,-2.5)--(2.5*\s,2.5);
    \draw[dashed] (17.5*\s,-2.5)--(17.5*\s,2.5);
    \draw[dashed] (2*\s,-1)--(14*\s,-1);
    \draw[dashed] (9*\s,0)--(17*\s,0);
    \draw[dashed] (14*\s,1)--(18*\s,1);
    %\node[] at (-3,2) {\LARGE $e_j(W)=2$};
    %\node[] at (-2.4,3) {\LARGE $e=3$};
    %\node[] at (-3,4) {\LARGE $e_i(W)=4$};
\end{tikzpicture}
}
\end{center}
\caption{Illustration on how the elevation multiset of the word $W=LL(RL)^3R(RL)^2R(RL)RRLL$ can be written as in Lemma \ref{multisetvalues}. For this word, $a=b=2$, $m=4$, and $(k_1,k_2,k_3,k_4)=(3,2,1,0)$. The three different parts of $E(W)$ are separated with vertical dashed lines. The horizontal dashed lines illustrate how the elevations $i-2$ for $1\leq i\leq 3$ appear $k_i+1+k_{i+1}$ times: first $k_i$ times ``at the peaks'' of $(RL)^{k_i}$ (black dots), then once more and finally $k_{i+1}$ times ``after the peaks'' of $(RL)^{k_{i+1}}$ (white circles). Note that there are no white circles along elevation 1 (topmost horizontal line), because $k_4=0$.}
\label{figuremultisetexample}
\end{figure}
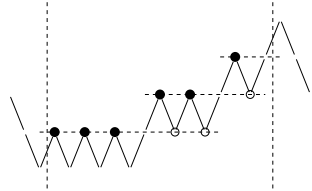

\end{lemma}

\white
\begin{proposition}
\label{unique-red}
Every equivalence class of balanced words contains a unique reduced word, which is the minimal word of the equivalence class.
\end{proposition}
\begin{proof} Existence: The minimal word of an equivalence class is reduced, by Lemma \ref{minimal-red}.

Uniqueness: Assume that $X$ and $Y$ are reduced words with $X\sim Y$. It will be shown that $X=Y$. By Proposition \ref{reduced-equiv}, both $X$ and $Y$ are of the form (iii). By Lemma \ref{multisetvalues}, the smallest and largest elevation values of a word of this form are $-a$ and $b$, respectively. The assumption $X\sim Y$ implies $E(X)=E(Y)$ by Lemma \ref{equivmultiset}, which means that $X$ and $Y$ share the parameters $a$ and $b$. Thus $X=L^aX'L^b$ and $Y=L^aY'L^b$, where $X'=(RL)^{k_1}R\cdots (RL)^{k_m}R$ and $Y'=(RL)^{l_1}R\cdots (RL)^{l_m}R$ for some $k_i\geq 0$ and $l_i\geq 0$, where $1\leq i\leq m$. (Note that the number of factors $(RL)^kR$ is $m=a+b$ for both $X'$ and $Y'$.) It suffices to show that $X'=Y'$.

Using the expression for $E(X)=E(Y)$ from Lemma \ref{multisetvalues}, and especially the fact that the multiplicities of the values in the third part agree, leads to the system of equations
$$\begin{cases}k_1=l_1\\k_i+1+k_{i+1}=l_i+1+l_{i+1}, \quad & 1\leq i\leq m-1\\k_m=l_m.\end{cases}$$
It is easy to see by back-substitution that $k_i=l_i$ for all $1\leq i\leq m$. Thus $X'=Y'$ and therefore $X=Y$. This shows the uniqueness.
\end{proof}

\paragraph{Finding the minimal word:} The reduction algorithm (presented below) takes in a balanced word $X$ and produces a reduced word in the equivalence class of $X$; by Proposition \ref{unique-red}, this reduced word is the minimal word of the equivalence class of $X$. The key idea of the algorithm is to look for certain kinds of subwords and perform swaps until no such subwords exist.

\paragraph{Reduction algorithm:} Let $X$ be a balanced word. The reduction algorithm is defined using a recursive sequence $(X_0,X_1,\dots)$ of words, starting with $X_0=X$. For $i\geq 0$, proceed as follows. If $X_i$ is reduced, the algorithm terminates with output $X_i$. If $X_i$ is not reduced, find the leftmost occurrence of a subword of type $UD$ where $U$ is an upper prime and $D$ is a lower prime; then $X_i=W_1UDW_2$ for some words $W_1,W_2$. Now let $X_{i+1}=W_1DUW_2$. 

\white
\begin{proposition}
The reduction algorithm terminates.
\end{proposition}
\begin{proof}
For each $i$, The words $X_i$ and $X_{i+1}$ are related by a swap, which implies that every $X_i$ belongs to the  equivalence class of $X$. By Lemma \ref{reductionlemma}, $X_{i+1}<X_i$ for each $i$, so all $X_i$ are distinct. However, the equivalence class of any word is finite by Corollary \ref{finite}. These observations together imply that the algorithm must terminate.
\end{proof}

\paragraph{Note:} Equivalence classes of prime words can be found by listing all the prime words of given length, and then applying the reduction algorithm to each word. An equivalence class consists of prime words that give the same output.

\paragraph{Example:} The following two tables list all the words in the equivalence classes of upper and lower primes  of length at most 10. Each row corresponds to an equivalence class. The words are listed in alphabetical order, so the minimal word always comes first. The parentheses are added to highlight the differences between words in the equivalence class.
%\white

\begin{figure}[H]
\begin{center}
\begin{tabular}{|c|}
    \hline
    Upper primes\\
    \hline
    $RL$\\
    \hline
    $RRLL$\\
    \hline
    $RRLRLL$\\
    \hline
    $RRRLLL$\\
    \hline
    $RRLRLRLL$\\
    \hline
    $RR(LR)(RL)LL$,\, $RR(RL)(LR)LL$\\
    \hline
    $RRRLRLLL$\\
    \hline
    $RRRRLLLL$\\
    \hline
    $RRLRLRLRLL$\\
    \hline
    $RR(LR)(LR)(RL)LL$,\, $RR(LR)(RL)(LR)LL$,\, $RR(RL)(LR)(LR)LL$\\
    \hline
    $RR(LR)(RL)(RL)LL$,\, $RR(RL)(LR)(RL)LL$,\, $RR(RL)(RL)(LR)LL$\\
    \hline
    $RR(LR)(RRLL)LL$,\, $RR(RRLL)(LR)LL$\\
    \hline
    $RRRLRLRLLL$\\
    \hline
    $RRR(LR)(RL)LLL$,\, $RRR(RL)(LR)LLL$\\
    \hline
    $RRRRLRLLLL$\\
    \hline
    $RRRRRLLLLL$\\
    \hline
\end{tabular}    
\end{center}
    \caption{Equivalence classes of upper primes of length $\leq 10$.}
    \label{listupper}
\end{figure}

\begin{figure}[H]
\begin{center}
\begin{tabular}{|c|}
    \hline
    Lower primes \\
    \hline
    $LR$\\
    \hline
    $LLRR$ \\
    \hline
    $LLLRRR$\\
    \hline
    $LLRLRR$\\
    \hline
    $LLLLRRRR$\\
    \hline
    $LLLRLRRR$\\
    \hline
    $LL(LR)(RL)RR$,\, $LL(RL)(LR)RR$\\
    \hline
    $LLRLRLRR$\\
    \hline
    $LLLLLRRRRR$ \\
    \hline
    $LLLLRLRRRR$\\
    \hline
    $LLL(LR)(RL)RRR$,\, $LLL(RL)(LR)RRR$\\
    \hline
    $LL(LLRR)(RL)RR$,\, $LL(RL)(LLRR)RR$ \\
    \hline
    $LLLRLRLRRR$ \\
    \hline
    $LL(LR)(LR)(RL)RR$,\, $LL(LR)(RL)(LR)RR$,\, $LL(RL)(LR)(LR)RR$ \\
    \hline
    $LL(LR)(RL)(RL)RR$,\, $LL(RL)(LR)(RL)RR$,\, $LL(RL)(RL)(LR)RR$ \\
    \hline
    $LLRLRLRLRR$ \\
    \hline
\end{tabular}    
\end{center}
    \caption{Equivalence classes of lower primes of length $\leq 10$.}
    \label{listlower}
\end{figure}

\pagebreak

\paragraph{Example:} Let $U=RL$ and $D=LR$. One equivalence class of upper primes of length 12 consists of six words that are of the type $RRZLL$ where $Z$ is a product of four prime words: two copies of $U$ and two copies of $D$. Note that there are exactly $\binom{4}{2}=6$ ways to arrange these four prime words. The words are as follows:
\begin{itemize}
    \item $W_{DDUU}=RR(LR)(LR)(RL)(RL)LL$
    \item $W_{DUDU}=RR(LR)(RL)(LR)(RL)LL$
    \item $W_{DUUD}=RR(LR)(RL)(RL)(LR)LL$
    \item $W_{UDDU}=RR(RL)(LR)(LR)(RL)LL$
    \item $W_{UDUD}=RR(RL)(LR)(RL)(LR)LL$
    \item $W_{UUDD}=RR(RL)(RL)(LR)(LR)LL$.
\end{itemize}

Figure \ref{figurealmostcomplete} illustrates how the words are related by swaps: each edge corresponds to a swap, and the directed edges are the special swaps from the reduction algorithm. The graph is almost complete, with only two edges missing. One of the missing edges is between $W_{UUDD}$ and $W_{DUDU}$, and the other one between $W_{DDUU}$ and $W_{UDUD}$. It can be seen that neither of these two pairs is related by a swap, by checking all the possible subwords $FG$ where $F$ and $G$ are balanced words.

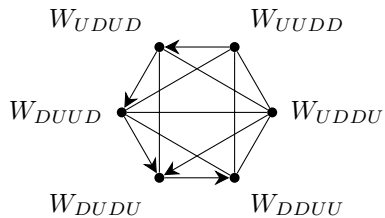
\begin{figure}[H]
\begin{center}
\begin{tikzpicture}
    \node[circle,fill=black,scale=0.4] at (1,0) (1) {};
    \node[circle,fill=black,scale=0.4] at (0.5,0.87) (2) {};
    \node[circle,fill=black,scale=0.4] at (-0.5,0.87) (3) {};
    \node[circle,fill=black,scale=0.4] at (-1,0) (4) {};
    \node[circle,fill=black,scale=0.4] at (-0.5,-0.87) (5) {};
    \node[circle,fill=black,scale=0.4] at (0.5,-0.87) (6) {}; 
    \draw[-{Stealth[length=2mm, width=2mm]}] (2)--(3);
    \draw[-{Stealth[length=2mm, width=2mm]}] (3)--(4);
    \draw[-{Stealth[length=2mm, width=2mm]}] (4)--(5);
    \draw[-{Stealth[length=2mm, width=2mm]}] (5)--(6);
    \draw[-{Stealth[length=2mm, width=2mm]}] (1)--(5);
    \draw (6)--(1)--(2)--(6)--(4)--(1)--(3)--(5);
    \draw (2)--(4);
    \node[right = .5mm of 1] {$W_{UDDU}$};
    \node[above right = .5mm of 2] {$W_{UUDD}$};
    \node[above left = .5mm of 3] {$W_{UDUD}$};
    \node[left = .5mm of 4] {$W_{DUUD}$};
    \node[below left = .5mm of 5] {$W_{DUDU}$};
    \node[below right = .5mm of 6] {$W_{DDUU}$};
\end{tikzpicture}
\end{center}
\caption{Illustration on how the words of the equivalence class are related by swaps. There is an edge (directed or undirected) between words if and only if the words are related by a swap. The directed edges indicate the swaps that occur when the reduction algorithm is applied.}
\label{figurealmostcomplete}
\end{figure}


\begin{thebibliography}{9}
\bibitem{paul}
Paul Terwilliger (1993) The Subconstituent Algebra of an Association Scheme, Part III. Journal of Algebraic Combinatorics 2 (1993), 177-210.

\bibitem{junie}
Junie T. Go (2002) The Terwilliger Algebra of the Hypercube. European Journal of Combinatorics (2002), 399-429.
\end{thebibliography}
\end{document}